\documentclass[12pt]{amsart}
\usepackage{amssymb}
\usepackage{amsmath} 
\usepackage{verbatim}
\usepackage{hyperref}

\newcommand\A{\mathcal A}
\newcommand\Aff{\mathfrak A}
\newcommand\AND{\quad\text{and}\quad}
\newcommand\bal{\boldsymbol{\alpha}}
\newcommand\bep{\boldsymbol{\epsilon}}
\newcommand\BB{\textit{\textbf{B}}}
\newcommand\Bor{\mathcal B}
\newcommand\C{\mathbb C}
\newcommand\cb{\mathbf{c}}
\newcommand\E{\mathbb E}
\newcommand\DD{\mathbb D}

\newcommand\ep{\varepsilon}
\newcommand\Exp{\textit{Exp}}
\newcommand\F{\mathcal F}
\newcommand\gb{\mathbf{g}}
\newcommand\Gb{\mathbf{G}}
\newcommand\gm{g}
\newcommand\HH{\mathbb H}
\newcommand\im{\mathfrak{i}\,}
\newcommand\Laa{\sigma}
\newcommand\Lap{\mathfrak{L}}
\newcommand\Ls{\Lambda}
\newcommand\lt{\mathfrak{l}}
\newcommand\mm{\mathsf m}

\newcommand\Mbf{\mathbf M}
\newcommand\N{\mathbb N}
\newcommand\nn{N}

\newcommand\Poiss{\Pi}
\newcommand\Prob{\mathbb P}
\newcommand\R{\mathbb R}
\newcommand\rad{\textsl{rad}}
\newcommand\ro{d}
\newcommand\rt{\mathfrak{r}}
\newcommand\sd{\mathfrak{s}}
\newcommand\T{\mathcal T}
\newcommand\ut{\mathfrak{u}}
\newcommand\vt{\mathfrak{v}}

\newcommand\wt{\widetilde}
\newcommand\ww{W}
\newcommand\uno{\mathbf 1}

\newtheorem{theorem}{Theorem}%[section]
\numberwithin{theorem}{section}
\newtheorem{pro}[theorem]{Proposition}%[section]
\newtheorem{lem}[theorem]{Lemma}%[section] 
\newtheorem{cor}[theorem]{Corollary}%[section]
\theoremstyle{definition}
\newtheorem{dfn}[theorem]{Definition}%[section]
\newtheorem{rem}[theorem]{Remark}%[section]
\newtheorem{rems}[theorem]{Remarks}
\newtheorem{ques}[theorem]{Questions}%[section]
\begin{document}$\,$ \vspace{-1truecm}

\title{Notes on hyperbolic branching Brownian motion}

\author{\bf Wolfgang Woess}
\address{\parbox{.8\linewidth}{Institut f\"ur Diskrete Mathematik,\\ 
Technische Universit\"at Graz,\\
Steyrergasse 30, A-8010 Graz, Austria\\[-3pt]}}
\email{woess@tugraz.at}

\date{March 15, 2026} 

\begin{abstract}
Euclidean branching Brownian motion (BBM) has been intensively studied
during many decades by renowned researchers.
BBM on hyperbolic space has received less attention. A
profound study of Lalley and Sellke (1997) provided insight
on the recurrent, resp. transient regimes of BBM on the Poincar\'e
disk. In particular, they determined the Hausdorff dimension
of the limit set on the boundary circle in dependence on the
fission rate of the branching particles. In the present notes,
further features are exhibited.  The rates of the maximal
and minimal hyperbolic distances to the starting point are determined, as 
well as refined asymptotic estimates in the transient regime. 
The other main issues studied here concern the behaviour of the
empirical distributions of the branching population, as
time goes to infinity, and their convergence to an infinitely supported 
random  limit probability measure on the boundary.
\end{abstract}
\maketitle
       
\markboth{{\sf Wolfgang Woess}}
{{\sf Notes on hyperbolic branching Brownian motion}}

\baselineskip 15pt

%%%%%%%%%%%%%%%%%%%%%%%%%%%%%%%%%%%
\section{Introduction}

Euclidean branching Brownian motion on $\R$ has been an intensively studied topic.
A good reference for the earlier developments is the comprehensive monograph 
by {\sc Athreya and Ney}~\cite{AN}. In short, a particle performs Brownian motion for
an exponentially distributed time and then fissions in two new particles, each of which continues
its own Brownian motion independently of the other for exponential time, then in turn fissions in two,
and so on. Primary object of studies has been the evolution of the random population at time 
$t$, as $t \to \infty$. One of the many interesting features concerns the position of the rightmost
particle at time $t$.  {\sc McKean}~\cite{McK} showed that the distribution of
the rightmost particle is directly linked with the travelling wave equation studied in the famous work of 
{\sc Kolmogorov, Petrovski and Piskunov}~\cite{KPP}.

Similarly, there is the study of branching random walk (BRW) in discrete time, where particles 
evolve according to a Galton-Watson process: at the points of fission, the new particles perform 
independent steps according to a given random walk. Branching random walk on $\R$ appeared, 
for example, in the monograph by {\sc Harris}~\cite[\S III.16]{Ha}. 
Both models have evolved significantly, and both are examples of tree-indexed Markov processes
in the sense of {\sc Benjamini and Peres}~\cite{BP}, where the linear time 
is replaced by a -- possibly random -- tree (discrete or continuous).

Here, the focus is on branching Brownian motion (BBM) 
on the hyperbolic disk. That is, the motion of the particles follows the continuous time 
Markov process whose infinitesimal generator is $\frac12 \Lap$, where $\Lap$ is the Laplace-Beltrami
operator of the unit disk equipped with the Poincar\'e metric (or equivalently, 
the upper half plane model). It has not been as widely considered as several other branching
random processes. A very significant in-depth study is due to {\sc Lalley and Sellke}~\cite{LS}. They 
considered hyperbolic BBM where the random times between successive fissions are exponentially
distributed with parameter $\lambda > 0$. There is a phase transition: for $\lambda \le 1/8$
(transient phase), each compact subset of $\DD$ is no more visited by the branching population
from some random time onwards, while for $\lambda > 1/8$ (recurrent phase) each open subset 
of $\DD$ continues to be visited by the population. The main results of \cite{LS} concern the random
\emph{limit set} $\Ls$, that is, the set of accumulation points of the trace of the population on the 
boundary $\partial \DD$, the unit circle. It is the full circle when $\lambda > 1/8$, while it 
is a perfect set with Hausdorff dimension $(1-\sqrt{1-8\lambda}\,)/2$ when $\lambda \le 1/8$.

An analogous phenomenon was exhibited for branching random walk on regular trees, resp. free groups,
by {\sc Liggett}~\cite{Li} and {\sc Hueter and Lalley}~\cite{HL}. This was extended to BRW on 
free products of groups by {\sc Candellero, Gilch and M\"uller}~\cite{CGM}, and very  recently  to 
BRW on hyperbolic groups by  {\sc Sidoravicius, Wang and Xiang}~\cite{SWX} and 
subsequently to BRW on relatively hyperbolic groups by {\sc Dussaule, Wang and Yang}~\cite{DWY}.
For BRW on finitely generated groups, resp. transitive graphs, this study has been accompanied by the 
investigation of the \emph{trace,} that is, the subgraph spanned by all points visited by the BRW.
See {\sc Benjamini and M\"uller}~\cite{BM}, {\sc Candellero and Roberts}, \cite{CR},
{\sc Gilch and M\"uller}~\cite{GiM}, {\sc Hutchcroft}~\cite{Hu} as well as \cite{SWX} and \cite{DWY}.

Returning to hyperbolic BBM, in the present notes several results are added to the 
profound study \cite{LS}.
The initial parts are presented in a rather broad way.
This concerns, in particular, the construction of the underlying time tree (Yule tree)  in \S 
\ref{sec:yule} and the chosen notation, which may differ slightly from the previous mainstream 
and is used
throughout sections \ref{sec:max}--\ref{sec:properties}. 
The introductory part (\S \ref{sec:hyp}) on (non-branching) hyperbolic BM includes, among other, 
a tail estimate of its maximal distance from the starting point within the time-interval 
$[0\,,\,1]$ (Proposition \ref{pro:max}), with which the experts seemingly have not been familiar. 
The construction of hyperbolic BBM is explained in 
\S \ref{sec:hypBBM} and related with a branching random walk on the affine group, a fact which might
merit further attention in future studies.

The main substance starts with \S \ref{sec:max}. Like in \cite{LS}, we make often use of the 
fact that in the upper half plane model, minus the logarithm of the imaginary part of hyperbolic BM
is one-dimensional Euclidean BM with drift $1/2$.
For the \emph{minimal and maximal hyperbolic distances} from the origin of a particle of hyperbolic 
BBM at time $t$, 
we show in Theorem \ref{thm:minmax1} that its rate is
the same as for the resulting Euclidean comparison process. 
Furthermore, in Theorem \ref{thm:minmax2} refined asymptotics are obtained 
in the transient regime, based on  
work of {\sc Roberts}~\cite{Ro} for Euclidean BBM.  

In \S \ref{sec:empdis} and \S\ref{sec:properties} we consider the \emph{empirical distributions} of
hyperbolic BBM: these are the normalised 
occupation measures of the population at the times $t > 0$. 
Following a suggestion by V. A. Kaimanovich, the asymptotic behaviour of these random distributions
was recently studied for BRW on graphs in parallel work of {\sc Kaimanovich and Woess}~\cite{KW}
and {\sc Candellero and Hutchcroft}~\cite{CH}. In the Euclidean setting of BRW and BBM, they had 
been studied since the mid 1960ies, see e.g. {\sc Ney}~\cite{Ney},  {\sc Stam}~\cite{St}, 
{\sc Kaplan and Asmussen}~\cite{KA}, {\sc Uchiyama}~\cite{U}, {\sc Biggins}~\cite{Bi1}, 
while \cite{KW}
and \cite{CH} are more relevant in a ``non-amenable'' setting. The latter is also present here, since
the spectral gap is responsible for the presence of the transient regime besides the recurrent one.

Variants of CLT-type asymptotics of the empirical distributions of hyperbolic BBM are 
derived once more via the Euclidean comparison process. The average rate of escape, that 
is, the rate of the average distance from the origin 
of the particles at time $t$, is shown in Theorem \ref{thm:escape} to coincide with the 
one of (non-branching) hyperbolic BM. In Theorem \ref{thm:weakconv} it is proved that 
the empirical distributions converge weakly
(on the closed disk) to a random limit distribution on the boundary, that is, the unit circle.
The limit distribution is infinitely supported (Theorem \ref{thm:atoms}). 
Some open questions are posed.

\smallskip

\noindent 
{\bf Acknowledgements.} The author acknowledges discussions with Nicola Kistler and Anton Wakolbinger
as well as helpful email exchange with John Biggins, Maury Bramson, Vadim Kaimanovich, 
Steve Lalley,  Sebastian M\"uller, Enzo Orsingher, Yuichi Shiozawa and Ofer Zeitouni. Several 
highly fruitful suggestions by a referee are also acknowledged.
%
%The author also 
%thanks the audience of his talks in Rome, Bristol, Innsbruck, Italy (Prisma online), 
%Vienna--Klosterneuburg and  Montpellier for the friendly feedback and stimulating discussions.

\section{The Yule tree}\label{sec:yule}

The introduction of the continuous random population tree described in this section goes back to
{\sc Yule}~\cite{Y}.

Consider the binary tree $\{\lt,\rt\}^*$ consisting of all binary sequences
(words) $\vt = \sd_1\cdots \sd_n$ with $n \ge 0$ and $\sd_k \in \{\lt,\rt\}$.\footnote{To 
avoid confusion with numbers, we avoid the binary symbols $0, 1$. The symbols $\lt$, $\rt$ 
and $\sd$ stand for ``left'', ``right'' and ``side'', respectively.} 
For $n=0$, this is the empty sequence $\bep$. 
If $\vt$ has the form $\vt=\ut\sd$ with $\sd \in \{\lt,\rt\}$, then the \emph{predecessor} of 
$\vt$ is $\vt'=\ut$, and $\vt$ is one of the two \emph{successors} of $\ut$. The 
edges of the tree are all $[\vt',\vt]$, where $\vt \ne \bep$. We augment this tree 
by an additional vertex $\bal$ which is only connected to $\bep$, so that
$\bep' = \bal$. %No length is assigned to $\bal$. 
Now consider a sequence of i.i.d. random variables $\ell_{\vt}\,$, $\vt \in \{\lt,\rt\}^*$
having exponential distribution with parameter $\lambda > 0$. 
For the moment, we can realise it for example on the  product space 
$$
\bigl(\Omega^{\text{Yule}}\,, \A^{\text{Yule}}\,,\Prob^{\text{Yule}}\bigr) 
= \bigotimes_{\vt \in \{\lt,\rt\}^*} 
\bigl( \R_+, \Bor_{\R_+}, \Exp_{\lambda}\bigr)_{\vt}\,,
$$
where each factor is a copy of the probability space consisting of the 
positive real half-line with the Borel sigma-algebra and the exponential 
distribution with parameter $\lambda$, so that $\ell_{\vt}$ is the projection
on the $v$-coordinate. Later on, it will be embedded into a bigger probability
space.

We then get a \emph{random tree} $\T$, a $1$-complex where each edge $[\vt',\vt]$ is 
an interval with the random length $\ell_{\vt}\,$, and we write the edge as 
\begin{equation}\label{eq:edge}
[\vt',\vt] = \{\tau = {{_s}\vt} : 0 \le s \le \ell_{\vt}\,\},
\end{equation}
where ${{_s}\vt}$ is the point in the interval
at distance $s$ from $\vt'$. Then the length (= distance from $\bal$) of a vertex 
$\vt$ is defined recursively by $|\bal|=0$ and $|\vt| = |\vt'| + \ell_{\vt}$, and the length of 
$\tau = {{_s}\vt} \in  [\vt^-,\vt]$ is $|\tau| = |\vt^-| + s$.

This is the Yule tree. It is interpreted as a \emph{genealogical tree}, where $\bal$
is the ``ancestor'' at time $0$, and thinking of it as a particle, after time 
$\ell_{\bep}$ it fissions in two particles. The timelines of the new particles
are the edges $[\bep,\lt]$ and $[\bep,\rt]$, respectively, and after times $\ell_{\lt}\,$,
resp. $\ell_{\rt}\,$, each of them fissions again in $2$ particles. Recursively, a particle
at the end of its timeline $[\vt',\vt]$ fissions in two, whose new timelines are the edges
$[\vt,\vt\lt]$ and $[\vt,\vt\rt]$, respectively.

The \emph{population} at time $t \ge 0$ is the set 
$\T(t) = \T_{\bep}(t)$ 
of all particles ($\equiv$ vertices or 
interior points on the edges) at distance $t$ from $\bal$. We write 
$$
\nn(t) = \nn_{\,\bep}(t) = |\T(t)|
$$
for their number. Note that by continuity of the exponential distribution, at any fixed time 
$t \ge 0$, with probability $1$ there is at most one vertex $v$ of $\T$ with $|v|=t$. 
That is, no two fissions take place simultaneously. %\fbox{More justification needed?}

For any $\ut \in \{\lt,\rt\}^*$ let $\T_{\ut}$ be the subtree of $\T$ in which $\ut$ has the 
role of $\bep$
in the above description and $\ut'$ the role of $\bal$. In other words, it is spanned  by
the vertex set 
$$
\{\ut'\} \cup \bigl\{ \ut\vt : \vt \in \{\lt,\rt\}^* \bigr\}. 
$$
We write $\T_{\ut}(t)$ for the associated part of the population, that is, the set of elements of 
$\T_{\ut}$ at distance $t$ from $\ut'$, and $\nn_{\ut}(t)$ for their number.

For any subset $U \subset \{\lt,\rt\}^*$ of vertices which is prefix-free (that is, no
element of $U$ is an initial part of another element in $U$),  the trees $\T_{\ut}\,,\; \ut \in U$
are i.i.d. In particular, all the generation sizes $\nn_{\ut}(t)$, $\ut \in \{\lt,\rt\}^*$,
have the same distribution %$\mu_t$ 
on $\N$. The following is well-known; %(even for more general models); 
we provide a short proof.

\begin{lem}\label{lem:expected} For any $t \ge 0$, the population size at
time $t$ has geometric distribution:
$$
\Prob[\nn(t)=n] = e^{-\lambda t}\,(1- e^{-\lambda t})^{n-1}\,,\quad n \in \N.
$$
In particular, the expected population size is $\;\E\bigl(\nn(t)\bigr) = e^{\lambda t}$.
\end{lem}

\begin{proof}
We have 
$$
\nn_{\,\bep}(t) = \begin{cases}  1\,,&\text{if }\; \ell_{\bep} \ge t\\
                \nn_{\lt}(t-\ell_{\bep})+\nn_{\rt}(t-\ell_{\bep})\,,&\text{if }\;
                                 \ell_{\bep} < t.     
                  \end{cases}
$$
Therefore the characteristic function (in the variable $x$) is 
$$
\begin{aligned}
\varphi_{\nn(t)}(x) = \E\Bigl(\exp\bigl(\im x \,\nn_{\bep}(t)\bigr)\Bigr) 
&= e^{\im x}\,\Prob[\ell_{\bep} \ge t] 
+ \E\Bigl(\exp\bigl(\im x \,\nn_{\lt}(t-\ell_{\bep})+\im x \,\nn_{\rt}(t-\ell_{\bep})\bigr)\,
\uno_{[\ell_{\bep} < t]}\Bigr)
\\
&= e^{\im x -\lambda t} 
+ \int_0^t  \varphi_{\nn(t-s)}(x)^2\, \lambda e^{-\lambda s}\, ds\,.
\end{aligned}
$$
Just for the purpose of this proof, set 
$f(t) = e^{\lambda t}\, \varphi_{\nn(t)}(x)$.
Then the above equation transforms into the integral equation
$$
f(t) = e^{\im x} + \lambda \int_0^t f(s)^2 \,e^{-\lambda s} \, ds.
$$
Thus, $f(t) = 1/(e^{-\im x} + e^{-\lambda t} - 1)$, and 
$$
\varphi_{\nn(t)}(x) = \frac{e^{\im x} \,e^{-\lambda t}}{1 - e^{\im x}(1-e^{-\lambda t})},
$$
which we recognise as the characteristic function of the geometric distribution with
parameter (success probability) $p = e^{-\lambda t}$.
\end{proof}

The population $\T(t)$ at time $t$ can be viewed as a \emph{cross-section} of $\T$ at level $t$.
For the tree viewed as a spatial generalisation of time, it can be interpreted to have 
a role analogous to the one of a stopping time.
If we cut the tree $\T$ along that cross-section, we obtain the subtree $\T(\vdash\! t)$ 
of all elements $\tau$ with $|\tau| \le t$. We let 
$V\bigl(\T(\vdash\! t)\bigr) = \{\lt,\rt\}^* \cap \T(\vdash\! t)$ be the random set consisting of 
those vertices of the  original binary tree which are part of $\T(\vdash\! t)\,$. 
By Lemma \ref{lem:expected}, $V\bigl(\T(\vdash\! t)\bigr)$ is 
finite with probability $1$.

The sigma-algebra $\F_t^{\text{Yule}}$ comprising the information inherent to the 
Yule tree up to and including time $t$ 
is generated by all random variables $\ell_{\vt}$ 
which intervene in the construction of $\T(\vdash\! t)$. 
\begin{comment}
This is a random collection, and more precisely,  
$$
\F_t = \{ A \in \A^{\text{Yule}}: A \cap [V\bigl(\T(\vdash\! t)\bigr) \subset V] \in \F_V\, 
\text{ for every finite }\;V \subset \{\lt,\rt\}^* \}
$$
Here as elsewhere we use the convention 
$$
[\text{logical expression on elements of }\; \Omega] 
= \{ \omega \in \Omega :  \text{the logical expression holds for }\omega \}. 
$$
\end{comment}
The following is well known.

\begin{pro}\label{pro:mart}
The process  $\;\bigl(\nn(t) \,e^{-\lambda t}\bigr)_{t \ge 0}$ is a martingale with respect to the
filtration $(\F_t^{\text{Yule}})_{t \ge 0}\,$, and there exists an almost surely positive 
random variable $\ww$ such that
$$
\lim_{t \to \infty} \nn(t) \,e^{-\lambda t} = \ww%_{\infty} 
\quad \text{almost surely and in } L^1.
$$
\end{pro}

\section{Hyperbolic disk and Brownian motion}\label{sec:hyp}

A standard model of two-dimensional hyperbolic space is the 
\emph{Poincar\'e disk},that is, the unit disk $\DD \subset \C$  
with the hyperbolic length element and resulting metric
\begin{equation}\label{eq:hypmetric}
ds =  \frac{2\sqrt{dx^2 + dy^2}}{1-|z|^2} \AND
\ro(z,w) = \log\frac{|1-z\bar w| + |z-w|}{|1-z\bar w| - |z-w|}.
\end{equation}
Its orientation preserving isometry group consists of all M\"obius transformations
of the form 
\begin{equation}\label{eq:moebius}
\gm z = \frac{az+ c}{\bar c z + \bar a}\,, \quad a, c \in \C, \; |a|^2-|c|^2 = 1\,.
\end{equation}
Together with the reflection $z \mapsto - \bar z$, it generates the full isometry group.
Of course, these transformations also act on the boundary of $\DD$, the unit circle $\partial \DD$.
We shall use the group $\Aff$ of all transformations as in \eqref{eq:moebius} which fix the
boundary point $1$, that is, $a+c \in \R$. It acts simply transitively on $\DD\,$. In particular,
For each $z_0 \in \DD$, there is a unique element $\gm_{z_0} \in \Aff$ which maps $0$ to $z_0$.
It is given by 
\begin{equation}\label{eq:gammaz0}
\gm_{z_0}z = \frac{(1-z_0)z + (z_0 - |z_0|^2)}{(\bar z_0 - |z_0|^2)z + (1-\bar z_0)}\,.
\end{equation}
The hyperbolic Laplace (or Laplace-Beltrami) operator in the variable $z = x + \im y$ is 
\begin{equation}\label{eq:hypLap}
\Lap = \frac{(1-|z|^2)^2}{4}\, \Bigl( \partial_x^2 + \partial_y^2\Bigr)\,. 
\end{equation}
It is self-adjoint on $L^2(\DD, \mm)$, where $\mm=\mm_{\DD}$ is
the hyperbolic measure,
\begin{equation}\label{eq:hypmeas}
d\mm(z) = \frac{4dz}{(1-|z|^2)^2}\,. 
%= 4 \cosh^4 \frac{\ro_{\HH}(z,0)}{2}\,dz\,.
\end{equation}

The infinitesimal generator of \emph{hyperbolic Brownian motion} $(B_t)_{t \ge 0}$ is $\frac{1}{2}\Lap$.
(We remark that the factor $\frac12$ is used because in the analogous Euclidean setting we want 
that Brownian motion at time $t=1$ has standard normal distribution.)

One can see the latter
as a version of the standard two-dimensional Euclidean Brownian motion
slowed down as it gets close to the unit circle.
Responsible for slowing down is the factor $\frac{(1-|z|^2)^2}{4}$ which only depends on
the Euclidean distance of the current position $z$ from the circle. 

Since the hyperbolic Laplacian commutes with all hyperbolic isometries, 
hyperbolic Brownian motion (BM) is invariant under the latter: if $\gm$ is as in \eqref{eq:moebius}
and $(B_t)$ is (a version of) hyperbolic BM starting at $z_0 \in \DD$ then $(\gm B_t)$ is 
(a version of) hyperbolic BM starting at $\gm z_0\,$.
Equivalently, the heat kernel  $p_t(\cdot,\cdot)$
with respect to $\mm$ associated with $\frac12 \Lap$ is invariant under the diagonal
actions of every hyperbolic isometry. We write $P_t$ for the associated transition 
operator of hyperbolic BM: for any measurable set $K \subset \DD$
and  $z \in \DD$, resp., measurable function $f: \DD \to \R$,
\begin{equation}\label{eq:Pt}
P_t(z,K) = \int_{K} p_t(z,w)\, d\mm(w) \AND
P_tf(z) = \int_{\DD} p_t(z,w)f(w)\, d\mm(w)\,,
\end{equation}
whenever that integral is well defined.

Now let $(\Omega_b\,, \A_b\,, \Prob_b)$ be a suitable probability space on which 
hyperbolic BM starting at the origin is defined. With starting point $0$, as $t \to \infty$, 
the process converges almost surely to a $\partial \DD$-valued random variable $B_{\infty}\,$.
The distribution of $B_{\infty}\,$, being rotation invariant, is equidistribution on the unit circle:
$d\xi = \frac{1}{2\pi} e^{\im\phi} \,d\phi$, where $d\phi$ is the Lebesgue measure
on $[-\pi\,,\,\pi]$. (The elements of $\partial \DD$ are denoted $\xi, \eta$, etc.)

When the starting point is $z_0$ then, working on the same probability space, (a model of) 
the limit random variable is $\gm_{z_0} B_{\infty}\,$. The density of its distribution $\nu_{z_0}$ 
with respect to $d\xi$ is the \emph{Poisson kernel}
\begin{equation}\label{eq:poisson}
\Poiss(z_0,\xi) = \frac{1 - |z_0|^2}{|\xi - z_0|^2} \qquad (z_0 \in \DD\,,\; \xi 
 = e^{\im \phi} \in \partial \DD).
\end{equation} 

Since $\ro( B_t, o) \to \infty$ almost surely, we are also interested in
the speed. It is linear, and to understand it, it may be useful to pass to
two other models of hyperbolic space.

The second one, besides the disk, is the \emph{upper half plane model} 
$\HH = \{ u + \im v : u,v \in \R, v > 0\}$.
The metric $\ro = \ro_{\DD}$ of \eqref{eq:hypmetric} is transported to the
metric $\ro_{\HH}$ via the M\"obius map from $\DD$ to $\HH$
\begin{equation}\label{eq:DtoH}
z \mapsto \im \frac{1+z}{1-z}.
\end{equation}
It maps $0$ to $\im$, and the boundary points $-1$ to $0$ and $1$ to $\im\infty$,
and
\begin{equation}\label{eq:hmetric}
\ro_{\HH}(z,w) = \log\frac{|z-\bar w| + |z-w|}{|z-\bar w| - |z-w|}.
\end{equation}
We shall often switch back and forth between $\DD$ and $\HH$, and will mostly use  
unified notation $o$ for our origin, that is, $o = 0$ in $\DD$ and $o=\im$ in $\HH$.
We also remark here that in the upper half plane model, the group $\Aff$ of all transformations
that fix the boundary point $1$ in $\DD$ (that is, $a+c \in \R$ for $a, c$ as in \eqref{eq:moebius}) 
becomes, via conjugation with the map \eqref{eq:DtoH},  the \emph{affine group} of all transformations
$$
z \mapsto az + b\,, \quad a> 0\,,\; b \in \R \quad (z \in \HH).
$$
(The $a$ here is not the same as in \eqref{eq:moebius}.)

In the coordinates $(u,v) \in \HH$, the 
hyperbolic Laplacian becomes $v^2 (\partial^2 u + \partial^2 v)$.
Then we make one more change of variables,
setting $w = \log v$ to obtain the \emph{logarithmic model,} where now $(u,w) \in \R^2$ and
the hyperbolic Laplacian becomes
$$  %\begin{equation}\label{eq:logmodel}
\Lap = e^{2w} \partial^2_u + \partial^2_w - \partial_w\,.
$$ %\end{equation}
Its projection on the vertical coordinate $w$ is $\partial^2_w - \partial_w$, so that 
$\frac12(\partial^2_w - \partial_w)$ is the infinitesimal generator of 
one-dimensional Euclidean Brownian motion with drift $-1/2$. Tracing this back to the
upper half plane and disk models, writing $B_t^{\HH}$ and $B_t^{\DD}$ for hyperbolic
Brownian motion in the respective coordinates, we have 
\begin{equation}\label{eq:im}
- \log \bigl(\Im B_t^{\HH}\bigr) = \tfrac{1}{2}t + \beta_t\,,
\end{equation} 
where $(\beta_t)_{t \ge 0}$ is standard Euclidean BM. In particular,
when $B_0^{\HH}=\im$, resp. $B_0^{\DD}=0$,
\begin{equation}\label{eq:normal}
\log \bigl(\Im B_t^{\HH}\bigr)  = \log \frac{1-|B_t^{\DD}|^2}{|1-B_t^{\DD}|^2}
\sim N\bigl(-\tfrac12 t, t\bigr),
\end{equation}
where (as usual) $\Im$ denotes the imaginary part, $\sim$ means ``has distribution'' 
and $N(a,s^2)$ is the normal distribution with mean $a$ and variance $s^2$.
From this we get the following, where we can omit the superscript referring to the respective model.

\begin{lem}\label{lem:speed}  The following central limit theorem and 
rate of escape results hold, as $t \to \infty\,$:
$$
%\begin{aligned}
\frac{\ro(B_t,B_0)- \frac12 t}{\sqrt{t}} \to \mathcal{N}(0,1) \quad \text{in law, and}\;\quad 
\frac{\ro(B_t,B_0)}{t} \to \frac12 \quad \text{almost surely.}
%\end{aligned}
$$
\end{lem}

\begin{proof}
As stated above \eqref{eq:poisson}, in terms of the disk model, it is well known that $B_t^{\DD}$ 
converges almost surely to a limit random variable $B_{\infty}^{\DD}$  with continuous distribution.
Taking this to the upper half plane model, since $\Prob[B_{\infty}^{\HH} = \im \infty]=0$, 
we get $B_{\infty}^{\HH} \in \R$ and $\Im B_t^{\HH} \to 0$ almost surely.  
Now, if $z = x + \im y \in \HH$
then an easy computation with the hyperbolic metric in $\HH$ as in \eqref{eq:hmetric} shows that 
\begin{equation}\label{eq:rhodiff}
\ro_{\HH}(x + \im y, \im) + \log y = %-  \ro_{\HH}(\im y, \im) = 
\log\frac12 \Bigl(1 + |z|^2 + \sqrt{(1 + |z|^2)^2-4y^2}\Bigr)
\end{equation}
We get that 
$$
\ro_{\HH}(B_t, \im) + \log (\Im B_t^{\HH}) \to \log\bigl(1+ (B_{\infty}^{\HH})^2\bigr) 
\quad \text{almost surely, as }\; t \to \infty, 
$$
and the limit is almost surely finite. Combining this with \eqref{eq:normal}
completes the proof.
\end{proof}

We shall need estimates of the tail behaviour of the random variables $\ro(B_t, o)$, $t >0$.
The heat kernel $p_t(o,z)$, $z \in \DD$ (resp., $\in \HH$) associated with $\frac12 \Lap$ 
only depends on $R= \ro(z,o)$, and it has a
uniform estimate in space and time, see 
{\sc Davies and Mandouvalos}~\cite[Thm. 3.1]{DM}
\footnote{Concerning the heat kernel, different normalisations of the Laplacian
are an ongoing source of small confusion. Analysts typically use the heat kernel for $\Lap$ as
given in \eqref{eq:hypLap}. Some of them also omit the factor $1/4$. In probability, we want that
Euclidean BM at time $t=1$ has variance $1$, so that we use the heat kernel for $\frac12 \Lap$. 
Although  not explicitly stated, \cite{DM} uses   
the first of these three options, so that here we had to replace
their $t$ by $t/2$.}.
We state here the upper bound in a way which is suited for our purpose: 
\begin{equation}\label{eq:heat1} 
\begin{aligned}
p_t(z,w) &\le 
\frac{\textsl{Const}}{\sqrt{1+R}} \,\,  \Psi\biggl(\frac{1+R}{t}\biggr)\,   
\exp\biggl(-\frac{(R+\tfrac{t}{2})^2}{2t}\biggr) 
\quad\text{for }\; R, t > 0\,, \quad \text{where}\\ 
R &= \ro(z,w) \AND \Psi(x) = \begin{cases} x^{3/2}\,,&0 \le x \le 1\,,\\
                                              x^{1/2}\,,&x \ge 1\,. 
                                \end{cases}
\end{aligned} 
\end{equation}

For the following proposition, apparently not present in the literature, 
the author acknowledges a suggestion by Yuichi Shiozawa.

\begin{pro}\label{pro:max} Let $\; \mathcal{M} = \max \{ \ro(B_t\,, B_0) : t \le 1 \}.$ Then there
is $\textsl{K} >0$ such that
for any $c > 0$ 
$$
\Prob[\mathcal{M} \ge c] \le 2 \, \Prob[\ro(B_1\,,B_0) \ge c] \le \textsl{K}\, 
\exp\biggl(-\frac{(c-\tfrac{1}{2})^2}{2}\biggr)\,.
%e^{-(c-1/2)^2/2}\,.
$$ 
\end{pro}

\begin{proof}
We work in the upper half plane model, and we may suppose without loss of generality 
that $B_0 = B_0^{\HH} = o \,(=\im)$. Here, we shall omit the sub- and superscripts $\HH$.
Recall from \eqref{eq:heat1} that in any model, $(B_t)_{t \ge 0}$ is isotropic, that is, its
transition kernel only depends on time and hyperbolic distance: 
%, or equivalently, that $B_0^{\HH} = \im$ in the upper half plane model. 
$\bigl(\ro(B_t\,,o)\bigr)_{t \ge 0}$ is a Markov process on 
the state space $[0\,,\,\infty)$ with continuous trajectories. 
In $\HH$, the point $\im e^{-c}$ is at hyperbolic distance $c$ from $o=\im$. Thus, using \eqref{eq:rhodiff},
for $s, t > 0$,
\begin{equation}\label{eq:gec}
\begin{aligned}
\Prob[\ro(B_{s+t},o) \ge c \mid \ro(B_s,o) = c] 
&= \Prob[\ro(B_{t},o) \ge c \mid B_0 =\im e^{-c}]\\ 
&\ge \Prob[ -\log(\Im B_t) \ge c \mid  -\log(\Im B_0) = c] \ge 1/2,
\end{aligned}
\end{equation}
because by \eqref{eq:normal}, if $\; -\log(\Im B_0) = c\;$ then
$\;-\log(\Im B_t) \sim N\bigl(c+\tfrac12 t, t\bigr)\,$.

Now consider the a.s. finite stopping time
$$
T_c = \inf \{ t > 0:  \ro(B_{t}\,,o) = c \}\,. 
$$
Then, using the strong Markov property, the fact that $\ro(B_{T_c}\,,o)=c$, and isotropy
of hyperbolic BM,
$$
\begin{aligned}
\Prob[ \mathcal{M} \ge c ] &= \Prob_0[ T_c \le 1 ]\\
&= \Prob[ \ro(B_1\,,o) \ge c\,, T_c \le 1 ] + \Prob[ \ro(B_1\,,o) < c\,, T_c \le 1 ]\\
&= \Prob[ \ro(B_1\,,o) \ge c ] 
+ \int_0^1\Prob[ \ro(B_1\,,o) < c  \mid \ro(B_{s}\,,o) =c]\,\,d\,\Prob[T_c=s]\\
\noalign{\noindent now applying \eqref{eq:gec}}\\[-19pt]
&\le \Prob[ \ro(B_1\,,o) \ge c ] 
+ \int_0^1\Prob[ \ro(B_1\,,o) \ge c  \mid \ro(B_{s}\,,o) =c]\,\,d\,\Prob[T_c=s]\\
&= 2\,\Prob[ \ro(B_1\,,o) \ge c ].
\end{aligned}
$$
This proves the first of the two inequalities. The second is going to be derived
from \eqref{eq:heat1}. Note that for $t=1$, we have to use $\Psi(x)=x^{1/2}$ in 
that heat kernel estimate. 

We return to the disk model.
We can express the hyperbolic measure of \eqref{eq:hypmeas} first in terms of
Euclidean polar coordinates $(r, \varphi)$ with $r <1$ and then replace 
$r = |z|$ by $R =  \ro(z,o) = \log \frac{1+r}{1-r}$. In the coordinates 
$(R,\varphi)$, we get
\begin{equation}\label{eq:hypmeas1}
d\mm(z) = \sinh R\,\, dR\, d\varphi \quad \text{in }\;\DD\,.
\end{equation}
Then
$$
\begin{aligned}
 \Prob[ \ro(B_1,o) \ge c] &= \int_{\{z \in \DD \,:\, \ro(z,o) \ge c\}} p_1(0,z)\,d\mm(z)\\
 &\le \textsl{Const} \int_{-\pi}^{\pi} \int_{c}^{\infty}
\,  \exp\biggl(-\frac{(R+\tfrac{1}{2})^2}{2}\biggr) 
%e^{-(R+\tfrac{1}{2})^2/2}
\, \frac{e^R - e^{-R}}{2}\,
dR\, d\varphi\\
&\le \textsl{Const}' \int_{c}^{\infty}
\exp\biggl(-\frac{(R-\tfrac{1}{2})^2}{2}\biggr) \, 
dR\\
&\le \textsl{Const}'' 
\exp\biggl(-\frac{(c-\tfrac{1}{2})^2}{2}\biggr) 
\end{aligned}
$$
for a suitable constant $\textsl{Const}''$. 
\end{proof}

\section{Hyperbolic branching Brownian motion}\label{sec:hypBBM}

We now construct hyperbolic branching Brownian motion. Whenever it is suitable,
we can switch between the different models of hyperbolic plane, but primarily
we have the disk model in mind.

We need a probability space $(\Omega, \A, \Prob)$ on which one can realise 
countably many  i.i.d. random variables $\ell_{\vt}\,$, $\vt \in \{\lt,\rt\}^*$ and,
independently of the latter, countably many independent hyperbolic Brownian motions 
$(B_t^{\vt})\,$, $\vt \in \{\lt,\rt\}^*$, each one starting at $0$.
If $(\Omega_b\,, \Bor_b\,, \Prob_b)$ is one probability space on which hyperbolic 
BM can be defined, then we can take the product space  
$$
\begin{gathered}
\bigl(\Omega\,, \A\,,\Prob\bigr) =
\bigl(\Omega^{\text{Yule}}\,, \A^{\text{Yule}}\,,\Prob^{\text{Yule}}\bigr)\otimes \bigl(\Omega_{\BB}\,, \A_{\BB}\,,\Prob_\BB\bigr)\,,
\quad\text{where}\\
\bigl(\Omega_{\BB}\,, \A_{\BB}\,,\Prob_{\BB}\bigr) = \bigotimes_{\vt \in \{\lt,\rt\}^*}
\bigl(\Omega_b,, \Bor_b\,,\Prob_b\bigr)_{\vt}\,,
\end{gathered}
$$ 
along with the corresponding projections.

With each fission point $\vt \in  \{\lt,\rt\}^*$, we associate a random element $\gb_{\vt} \in \Aff$, 
where $\Aff$ is the affine group, defined below \eqref{eq:moebius} for the disk model:
if $z_{\vt} = B_{\ell_{\vt}}^{\vt}o$ then $\gb_{\vt} = \gm_{z_{\vt}}$ as defined by \eqref{eq:gammaz0}. The 
$\gb_{\vt}\,$, $\vt \in \{\lt,\rt\}^*$, are i.i.d.

Furthermore, we take a starting point $z_0 \in \DD$. We shall often write 
$\Prob_{\! z_0}$ for the probability measure and $\E_{z_0}$ for the corresponding expectation
in order to remember the starting point.

Then hyperbolic branching Brownian 
motion assigns a random variable $\BB_{\tau}$
to every element $\tau$ 
of the Yule tree $\T$ (vertex or interior element of some edge) %recursively 
as follows, using the notation of \eqref{eq:edge}: 
\begin{itemize}
 \item[(A)] If the path from $\bal$ to $\vt \in \{\lt,\rt\}^*$ in our tree has 
 the vertices $\bal=\vt_0\,,\vt_1\,,\dots, \vt_k = \vt$
 (with $\vt_j' = \vt_{j-1}$) then for $\tau = {{_s}\vt} \in [\vt',\vt]$,
 $$
 \BB_{\tau} = \gm_{z_0}\gb_{\vt_1}\cdots\gb_{\vt_{k-1}}B_s^v\,.
 $$
\end{itemize}
This can also be described by  the following recursive construction:
\begin{itemize}
 \item [(B1)] At the ``ancestor'' $\bal$, 
 $$
 \BB_{\bal} = z_0\,.
 $$
\item[(B2)] If for a vertex $\vt \in \{\lt,\rt\}^*$, we already have $\BB_{\vt'} = z \in \DD$,  
then on the edge $[\vt',\vt]$, we continue with %$\gm_{z} B_t^{v}\,$, that is,
 $$
 \BB_{{{_s}\vt}} = \gm_{z} B_s^{\vt}\,,\quad 0 \le s \le \ell_{\vt}\,,
 $$
where the random element $g_z \in \Aff$ is given by \eqref{eq:gammaz0}, with $z$ in the place of $z_0\,$. 
\end{itemize}
In particular, in our construction, hyperbolic BBM starting at $z_0$ is the image under
$g_{z_0}$ of hyperbolic BBM starting at $0$.

\begin{rems} (a) In principle, we could choose any initial family of M\"obius transformations
such that for each $z_0 \in\DD$ there is precisely one $g_{z_0}$ mapping $o$ to $z_0$.  
For example (suggestion by Steve Lalley) one could take the map
$$
z \mapsto \frac{z + z_0}{\bar z_0z + 1}\,,
$$
which is as in \eqref{eq:moebius} with $a= 1/(1-|z_0|^2)^{1/2}$ and $c = z_0/(1-|z_0|^2)^{1/2}$.
When $z_0=0$ it is the identity map, while otherwise it fixes the 
diameter segment of the unit disk through the origin and $z_0\,$.
\\[3pt]
(b) However, the equivalence between the above two constructions (A) and (B1)+(B2) relies on the 
fact that the group $\Aff$ acts simply transitively, while this is not the case for the semigroup
generated by the mappings of (a).

\smallskip

A feature of the construction (A) is that in this way, hyperbolic BBM is governed by 
a branching random walk on the group $\Aff$. The underlying Galton-Watson tree is not random,
but the binary tree $\{\lt,\rt\}^*$. That is, each member of the corresponding population
fissions in precisely 2 children. The branching random walk starting at $g_{z_0}$ is then 
$(\gm_{z_0}\Gb_{\vt})_{\vt \in \{\lt,\rt\}^*}$, where 
\begin{equation}\label{eq:Gv}
%X_{\vt} 
\Gb_{\vt} =\gb_{\vt_1}\cdots\gb_{\vt_{k}}
\end{equation}
when the path from $\bep$ to $\vt$ is $[\bep = \vt_1\,, \vt_2\,,\dots, \vt_k = \vt]$. 
\end{rems}

An infinite ray $\pi$ in the Yule tree $\T$ is a line isometric with  $[0\,,\,\infty)$ 
which is spanned by a sequence of 
vertices $[\bal, \bep, ...]$ where every vertex $\vt \in \pi$ has precisely one
successor in $\pi$. Along $\pi$, the process $(\BB_{\tau})_{\tau \in \pi}$ is 
a hyperbolic BM starting at $z_0$ whose element at time $t$ is $\BB_{\tau}$ with 
$\tau = {{_s}\vt}$, where $\vt\in\pi$ is a 
vertex and $|{{_s}\vt}| = t$ as defined below \eqref{eq:edge}.
If we have two distinct rays $\pi$ and $\pi'$ then their \emph{confluent}
$\pi \wedge \pi'$ is the furthest vertex from $\bal$ shared by the two. 
Then the two hyperbolic Brownian motions along $\pi$ and $\pi'$ coincide up
to $\pi \wedge \pi'$ and thereafter continue independently. 

We repeat here an important fact which was already stated in the introduction.

\begin{pro}\label{pro:regimes}\cite{LS}
Transient regime: if $\lambda \le 1/8$ then for any compact $K \subset \DD$
$$
\Prob[ \text{there is }\; t > 0 : \BB_{\tau} \notin K\; 
\text{ for all $\tau \in \T$ with }\;|\tau| \ge t\,] =1.  
$$
Recurrent regime: if $\lambda > 1/8$ then for any non-empty open $U \subset \DD$
$$
\Prob[\,\text{for every }\; t > 0 : \BB_{\tau} \in U\; 
\text{ for some $\tau \in \T$ with }\;|\tau| \ge t \,] =1.  
$$
\end{pro}

\section{The maximal and minimal distances}\label{sec:max}

We first consider an issue which has been intensively studied for 
Euclidean BBM since its connection with the pioneering work of {\sc Kolmogorov, Petrovski and
Piskunov}~\cite{KPP} was elaborated by {\sc McKean}~\cite{McK}: the behaviour of the maximal and minimal 
hyperbolic distances from the starting point of the population at time $t$. 
%We only consider the basic result on the rate of escape. 
For details in the Euclidean case, see e.g. {\sc Bramson}~\cite{Br}, {\sc Bovier}~\cite{Bo} 
and {\sc Roberts}~\cite{Ro} and, in higher dimension, {\sc Kim, Lubetzky and Zeitouni}~\cite{KLZ}.

\begin{theorem}\label{thm:minmax1} Set 
$$
\mathsf{Max}_t = \max \bigl\{ \ro(\BB_{\tau},o) : \tau \in \T(t) \bigr\} \AND  
\mathsf{Min}_t = \min \bigl\{ \ro(\BB_{\tau},o) : \tau \in \T(t) \bigr\}\,.
$$
Then
$$
\frac{\mathsf{Max}_t}{t} \to r^* 
= \frac12+\sqrt{2\lambda}
\AND
\frac{\mathsf{Min}_t}{t} \to 
\begin{cases} r_* = \dfrac12-\sqrt{2\lambda}\,, &\text{if }\; 0 < \lambda \le 1/8\,,\\[3pt]
              0\,, &\text{if }\; \lambda \ge 1/8 
\end{cases}
$$
almost surely, as $t \to \infty\,$. As a matter of fact, 
$$
\limsup_{t \to \infty} \mathsf{Min}_t < \infty \quad \text{almost surely, when }\; \lambda > 1/8. 
$$
\end{theorem}

We shall need to prove this, in particular for the minimum, only in the recurrent regime
$\lambda > 1/8$, since in the transient regime a stronger result will be shown further below.
As a matter of fact, the transient regime is included in Theorem \ref{thm:minmax1} only in
order to provide a comprehensive picture for the rates of maximum and minimum.

We need known results. 
Let $\T$ be the Yule tree with parameter $\lambda$ as in \S \ref{sec:yule}, and let 
$(\boldsymbol{\beta}_{\tau})_{\tau \in \T}$ be the associated standard Euclidean BBM.

\begin{pro}\label{pro:Euc}
Set $\;
\mathsf{M}_t = \max \bigl\{ \boldsymbol{\beta}_{\tau} : \tau \in \T(t) \bigr\}\,$.
Then 
$$
\frac{\mathsf{M}_t - \sqrt{2\lambda}t}{\log t} \to - \frac{3}{\sqrt{8\lambda}}
\quad \text{in probability.}
$$
Furthermore
$$
\liminf_{t \to \infty} \frac{\mathsf{M}_t - \sqrt{2\lambda}t}{\log t} 
= - \frac{3}{\sqrt{8\lambda}}
\AND 
\limsup_{t \to \infty} \frac{\mathsf{M}_t - \sqrt{2\lambda}t}{\log t} = - \frac{1}{\sqrt{8\lambda}}
$$
almost surely. In particular, $\; \frac{1}{t}\mathsf{M}_t \to \sqrt{2\lambda}\;$ almost surely.
\end{pro}

\begin{proof} In the literature,  usually only the case $\lambda =1$ is considered for
the fissioning rate (i.e., the parameter of the exponential distribution of the 
edge lengths of the Yule tree). Convergence in probability is classical, and much more is
true. See
\cite{McK}, \cite{Br}, \cite[\S 3]{Bo}), and \cite[Thm. 2]{Ro} for the oscillations 
in a.s. convergence.

From here, one easily gets to our general $\lambda\,$: write $\lambda \T$ for the random tree 
where all the edge lengths are multiplied by $\lambda$.
This models the Yule tree with fissioning rate~$1$. 
Then by the scale-invariance of BM, 
$(\frac{1}{\sqrt{\lambda}}\,\boldsymbol{\beta}_{\lambda\tau})_{\tau \in \T}$
is standard Euclidean BBM, to which we can apply the result for rate $1$ and then go back. 
\end{proof}

A tool that will be used a few times is the well-known ``Many-to-one Lemma'' which we state here 
for hyperbolic BBM on $\DD$ (equivalently on $\HH$). 
See e.g. {\sc Harris and Roberts}~\cite{HaRo} or {\sc Shi}~\cite{Shi} for general versions;
the version stated here is taken from {\sc J.~Beresticky}~\cite{Ber} (adapted hyperbolically).

\begin{lem}\label{lem:many-one} For $t > 0$, let $C_{[0,t]}^{\DD}$ be the space of continuous
functions $[0,t] \to \DD$, and 
let $F: C_{[0,t]}^{\DD} \to \R$ be a bounded measurable function. For $\tau \in \T(t)$, 
consider the continuous function $s \mapsto \BB_{\tau(s)}\,$, where $\tau(s)$ is the element on
the geodesic in $\T$ from $\bal$ to $\tau$ with $|\tau(s)|=s$.
Then
$$
\E \Biggl(\, \sum_{\tau \in \T(t)} F(s \mapsto \BB_{\tau(s)}\,, s \le t)\! \Biggr) = 
e^{\lambda t} \, \E \bigl(F(s \mapsto B_s\,, s \le t)\bigr).
$$
\end{lem}
We shall typically use $F(f)$ to be a function of $\ro\bigl(f(t),o\bigr)$
or $\uno_K\bigl(f(t)\bigr)$, where $K \subset \DD$ is compact, or a function of 
$\max \bigl\{ \ro\bigl(f(s),o\bigr) :  s \le t \}$.

\begin{proof}[Proof of Theorem \ref{thm:minmax1}]
We start with the minimal distance, restricting ourselves to the recurrent regime $\lambda > 1/8$.
By \cite[Cor. 3]{LS}, with probability 1 there is an infinite ray
in $\T$ along which the entire trajectory of the Brownian particles remains in a compact set. 
This implies
the last statement of the theorem, and thus also that the rate of the minimal distance is $0$. 
For the transient regime, we refer to Theorem \ref{thm:minmax2} below.

\medskip

We now consider the rate of the maximal distance.
In the upper half plane model, 
the family $(-\log \Im \BB_{\tau}^{\HH})_{\tau \in \T}$ is 
one-dimensional Euclidean branching BM with drift $1/2$, see \eqref{eq:im}. Proposition \ref{pro:Euc}
yields
\begin{equation}\label{eq:maxreal}
\frac{1}{t}\max\bigl\{ -\log \Im \BB_{\tau} : \tau \in \T(t) \bigr\} 
= \frac12 + \frac{1}{t}\mathsf{M}_t \to r^*
\quad \text{almost surely, as }\; t \to \infty 
\end{equation}
Furthermore, again in the upper half plane model,
$\ro_{\HH}(z,o) \ge -\log \Im z$ by \eqref{eq:rhodiff}, which we have used already. 
This and \eqref{eq:maxreal} imply that
$$
\liminf_{t \to \infty} \frac{\mathsf{Max}_t}{t}  \ge r^* \quad \text{almost surely.}
$$
For what follows, it is preferable to return to the disk model. 
Let us assume, without loss of generality, that the starting point  is $o$. 
We choose a constant $C > r^*$.
By Lemma \ref{lem:many-one}
$$ 
\begin{aligned}
\Prob[ \mathsf{Max}_t \ge Ct] 
&= \Prob\Biggl(\, \bigcup_{\tau \in \T(t)} [ \ro(\BB_{\tau}\,,o) \ge Ct] \!\Biggr)\\ 
&\le \E \Biggl[\, \sum_{\tau \in \T(t)} \uno_{[Ct\,,\,\infty)}\bigl( \ro(\BB_{\tau}\,,o) \bigr) 
  \!\Biggr]
  = e^{\lambda t} \,\Prob[ \ro(B_t\,,o) \ge Ct]\,. 
\end{aligned}
$$ 
We need to estimate $\Prob[ \ro(B_t,o) \ge Ct]$.  This works similarly as in the 
last part of the proof of Proposition \ref{pro:max}.  We apply once more 
\eqref{eq:heat1} with the measure $\mm$ as  in \eqref{eq:hypmeas1}. Up to a change of the leading 
constant in \eqref{eq:heat1}, we can also use $\Psi(x) = x^{1/2}$ for  
$x \ge 1/2$ (instead of $x \ge 1$).
This applies to the range which we are now considering, namely $R \ge Ct \ge r^*t \ge t/2$,
in which case $R - t/2 > \sqrt{2\lambda}\,t$. For $t \ge 1/(2\lambda)$, we have 
$(R-t/2)/t \ge 1/\sqrt{t}$ and get the following estimate for 
(non-branching) hyperbolic BM in $\DD$ starting at $o$. 
$$
\begin{aligned}
 \Prob[ \ro(B_t,o) \ge C\, t] &= \int_{\{z : \ro(z,o) \ge Ct\}} p_t(o,z)\,d\mm(z)\\
 &\le \textsl{Const} \int_{-\pi}^{\pi} \int_{Ct}^{\infty}
\, \frac{1}{\sqrt{t}} \, \exp\biggl(-\frac{(R+\tfrac{t}{2})^2}{2t}\biggr) \, \frac{e^R - e^{-R}}{2}\,
dR\, d\varphi\\
&\le \textsl{Const}'  \int_{Ct}^{\infty}\frac{R - \tfrac{t}{2}}{t} \, 
\exp\biggl(-\frac{(R-\tfrac{t}{2})^2}{2t}\biggr)\, dR\\
&= \textsl{Const}'\,\,
\exp\biggl(-\frac{(C-\tfrac12)^2 t}{2}\biggr),
\end{aligned}
$$
(The constants are not assumed to be the same as in the proof of Proposition \ref{pro:max}.)
Combining our computations, we get 
$$
\Prob[ \mathsf{Max}_t \ge Ct] \le \textsl{Const}' \, \,
\exp\biggl(-\frac{\bigl((C-\tfrac12)^2 - 2\lambda\bigr) t}{2}\biggr)\,,\quad \text{for }\; 
t \ge \frac{1}{2\lambda}\,.
$$ 
Since $C > r^*$, we have $(C-\tfrac12)^2 - 2\lambda > 0$.
Consequently, at integer times
$$
\sum_{n=1}^{\infty} \Prob[ \mathsf{Max}_n \ge Cn] < \infty\,,
$$
and by the Borel-Cantelli Lemma, 
$$
\limsup_{n\to \infty} \frac{\mathsf{Max}_n}{n} \le C  \quad \text{almost surely.} 
$$
This holds for every $C > r^*$, whence in view of the lower bound
$$
\lim_{n\to \infty} \frac{\mathsf{Max}_n}{n} = r^* \quad \text{almost surely.}
$$
We now need to fill in the ``gaps'' for real $t \in (n-1\,,\,n)$, where 
$n$ runs through the positive integers. For this purpose, we first 
define
$$
\overline{\!\mathcal{M}} = \max \{ \ro(\BB_{\tau}\,,o) : \tau \in \T(\vdash\! 1)\}
$$
\emph{Claim 1.} For any $c > 0$, 
$$
\Prob[\,\overline{\!\mathcal{M}} > c] \le e^{\lambda} \, \Prob[\mathcal{M} >c]\,,
$$
where $\mathcal{M}$ is as in Proposition \ref{pro:max}, that is, the maximal distance
of ordinary hyperbolic BM from the starting point within the time interval $[0\,,\,1]$. 
\\[5pt]
\emph{Proof of Claim 1.} For each $\tau \in \T(1)$, consider the 
geodesic path in $\T$ from the root $\bal$ to~$\tau$. Along this timeline,
we see an ordinary hyperbolic BM starting at $0$, running up to time $1$. 
Let $\mathcal{M}_{\tau}$ be the maximal distance from $0$ of this BM. Then 
$$
\overline{\!\mathcal{M}} = \max \{ \mathcal{M}_{\tau} : \tau \in \T(1) \}.
$$
Claim 1 now follows from Lemma \ref{lem:many-one}.

Note that the event $[\,|\vt| \notin \N \;\text{for all}\; \vt \in \{\lt,\rt\}^*]$ has
probability $1$. We work on that event.
Now let $t > 0$ and $n = \lfloor t \rfloor$. 
Each $\tau \in \T(n)$ is the root of a single subtree $\T_{\tau} = \T_{\tau}(\vdash\! 1)$ 
of height $1$ within $\T$, whose end-vertices belong to $\T(n+1)$. Let 
$$
\overline{\!\mathcal{M}}_{\tau} = \max \{ \ro(\BB_{\theta}\,,\BB_{\tau}) : \theta \in \T_{\tau}\}
$$
Conditionally upon $\F_n = \F_n^{\text{BBM}}$, 
the random variables $\;\overline{\!\mathcal{M}}_{\tau}\,$, $\tau \in \T(n)$, are 
independent and (since the exponential distribution is memoryless) have the same 
distribution as the RV $\;\overline{\!\mathcal{M}}$ of Claim 1.
We clearly have 
$$
\mathsf{Max}_t \le \mathsf{Max}_n + \max \{ \,\overline{\!\mathcal{M}}_{\tau} : \tau \in \T(n) \}.
$$
The proof of the theorem will be complete if we can show that 
\begin{equation}\label{eq:maxmax}
\lim_{n \to \infty} \frac{1}{n} \max \{ \,\overline{\!\mathcal{M}}_{\tau} : \tau \in \T(n) \} = 0
\quad \text{almost surely.}
\end{equation}
Let $c > 0$. Then, using Lemma \ref{lem:many-one} and Claim 1,
$$
\begin{aligned}
\Prob\bigl[\max \{ \,\overline{\!\mathcal{M}}_{\tau} : \tau \in \T(n) \} > cn]
&\le \Prob\bigl[\,\overline{\!\mathcal{M}} > cn]\, e^{\lambda n}
\le \, \Prob[\mathcal{M} >cn]\,e^{\lambda(n+1)}\\
&\le \textsl{K}\, 
\exp\biggl( -\frac{(cn-\tfrac{1}{2})^2- 2\lambda(n+1)}{2}\biggr)\,
\end{aligned}
$$
by Proposition \ref{pro:max}.
Summing over all $n$, the resulting series converges for every $c> 0$, and once more, 
the Borel-Cantelli Lemma yields \eqref{eq:maxmax}.
\end{proof}

\begin{theorem}\label{thm:minmax2} In the transient regime $\lambda \le 1/8$,
$$
\frac{\mathsf{Max}_t - r^*t}{\log t} \to - \frac{3}{\sqrt{8\lambda}} \AND 
\frac{\mathsf{Min}_t - r_*t}{\log t} \to  \frac{3}{\sqrt{8\lambda}} 
$$
in probability, 
as $t \to \infty$,\footnote{This statement has been included thanks to a suggestion by a referee.} while 
$$
\limsup_{t \to \infty} \frac{\mathsf{Max}_t - r^*t}{\log t} = - \frac{1}{\sqrt{8\lambda}}
\AND 
\liminf_{t \to \infty} \frac{\mathsf{Max}_t - r^*t}{\log t} = - \frac{3}{\sqrt{8\lambda}}
$$
and 
$$
\limsup_{t \to \infty} \frac{\mathsf{Min}_t - r_*t}{\log t} =  \frac{3}{\sqrt{8\lambda}}
\AND 
\liminf_{t \to \infty} \frac{\mathsf{Min}_t - r_*t}{\log t} =  \frac{1}{\sqrt{8\lambda}}
$$
almost surely.
\end{theorem}

\begin{proof}
We work with $\HH$ and assume once more that the starting point is $o=\im$.

In the standard Euclidean case, symmetry of ordinary real BM implies that the minimum 
of BBM at time $t$ behaves like minus the maximum. Now recall once more  \eqref{eq:im}. We see
that the proposed behaviour is the one related to maximum and minimum of 
$-\log \Im \BB_{\tau}^{\HH}$. 
By \eqref{eq:rhodiff},
$$
-\log \Im \BB_{\tau}^{\HH} \le \ro_{\HH}(\BB_{\tau}^{\HH},\im) \le -\log \Im \BB_{\tau}^{\HH} 
+ |\BB_{\tau}^{\HH}|^2\,, 
$$
where the latter is the square of the absolute value in $\C$. 
Thus, combined with Proposition \ref{pro:Euc} the following will prove the theorem.
\\[5pt]
\emph{Claim 2.} \hspace{.9cm}
$ 
\dfrac{1}{\log t}\max \bigl\{ |\BB_{\tau}^{\HH}|^2 : \tau \in \T(t) \bigr\} = 0
\quad \text{almost surely, as }\; t \to \infty\,.
$
\\[5pt]
\emph{Proof of Claim 2.}
By \eqref{eq:im} and Proposition \ref{pro:Euc}, 
$$
\limsup_{t \to \infty} 
\frac{1}{\log t}\Bigl(\max\nolimits_{|\tau|=t}\log \Im \BB_{\tau}^{\HH}  + r_*t\Bigr)
= -\frac{1}{\sqrt{8\lambda}}
\quad \text{almost surely.}
$$
Since $r_* \ge 0$, this implies that for any $\ep > 0$,
\begin{equation}\label{eq:Pto1}
 \Prob\Bigl[ \sup\nolimits_{|\tau| \ge t} \Im \BB_{\tau}^{\HH} < e^{-r_*t}\, t^{-(1-\ep)/\sqrt{8\lambda}}
             \Bigr] \to 1 \quad \text{as }\; t \to \infty\,.
\end{equation}
We see that in Claim 2 we only need to consider  $|\Re \BB_{\tau}^{\HH}|^2$. 
We use the criterion that for a sequence of non-negative random variables, $X_n \to 0$ almost surely 
if and only if $\sup \{ X_k : k \ge n \} \to 0$ in probability as $n \to \infty$.
Thus, we work with integer times and take
$$
X_n = \frac{1}{\log n}\sup \{ |\Re \BB_{\tau}^{\HH}|^2 : n \le |\tau| < n+1 \},
$$
so that almost sure convergence to $0$ of $X_n$ implies the same for  
$
\frac{1}{\log t}\max_{|\tau| = t} |\Re \BB_{\tau}^{\HH}|^2. 
$
Now for any $C > 0$, 
$$
\Prob\bigl[ \sup\nolimits_{k \ge n} X_k > C^2 \bigr] \le 
\Prob \bigl[ \sup\nolimits_{|\tau| \ge n} |\Re \BB_{\tau}^{\HH}|^2 > C^2 \log n\bigr].
$$
We show that the latter probability tends to $0$ as $n \to \infty$, which will 
yield Claim 2.
It is 
$$
\begin{aligned}
\le \Prob\biggl[ \sup\nolimits_{|\tau| \ge n} \bigl|\Re \BB_{\tau}^{\HH}\bigr| > C \sqrt{\log n}\,,
\;\;
\sup\nolimits_{|\tau| \ge n} \Im \BB_{\tau}^{\HH} < e^{-r_*n}\, n^{-(1-\ep)/\sqrt{8\lambda}} \biggr]
&\ \\ + \ 
\Prob\biggl[ \sup\nolimits_{|\tau| \ge n} \Im \BB_{\tau}^{\HH} 
             \ge e^{-r_*n}\, n^{-(1-\ep)/\sqrt{8\lambda}}\biggr]&.
\end{aligned}
$$
By \eqref{eq:Pto1}, the second term tends to $0$ as $n \to \infty$.
The first term is bounded above by the probability that at some time $t \ge n$, hyperbolic 
BBM in $\HH$ visits the set 
$$
\begin{aligned}
D_n &= \{ z \in \HH : |\Re z| > R_n \,,\; \Im z < h_n\}\,, \quad \text{where}\\ 
R_n &= C\sqrt{\log n} \AND h_n = e^{-r_*n}\, n^{-(1-\ep)/\sqrt{8\lambda}}.
\end{aligned}
$$
We now borrow an argument from the proof of \cite[Prop. 4]{LS}. 
The mapping $z \mapsto \frac{z-\im}{z+\im}$ from $\HH$ to $\DD$ is the inverse of \eqref{eq:DtoH}.
It maps $D_n$ into the set  
$$
\wt D_n = \left\{ z \in \C : |z| < 1\,,\; \Bigl|z - \frac{h_n}{1+h_n}\Bigr| > \frac{1}{1+h_n}\,,\;
|\sin \arg z| < \frac{2R_n}{R_n^2+1} \right\}, 
$$
a region in the plane close to $1$ which is small for the Euclidean eye, since $R_n$ is large 
and $h_n$ is small. (The image of $D_n$ is slightly smaller than $\wt D_n$ on its ``sides''.)

Now consider the three circles $\mathcal{C}_{-1}\,$, $\mathcal{C}_0 $ and $\mathcal{C}_1\,$
with radius $1\big/\!\sqrt{R_n^2+1}$ which are tangent from inside to the unit circle at the points
$\dfrac{R_n-\im}{R_n+\im}$, $1$ and $\dfrac{R_n+\im}{R_n-\im}$, respectively. (The latter are 
the images of the boundary points $-R_n\,$, $\im\infty$ and $R_n$ of $\HH$.)
When $n$ is sufficiently large, $2h_n/(1+h_n) \le 1\big/\!\sqrt{R_n^2+1}\,$.
Then the union of the these circles (in fact only part thereof) separates the origin from the set 
$\wt D_n$. Thus, all particles of hyperbolic BBM in $\DD$ starting from $0$ which reach $\wt D_n$
must pass through that union. This and rotation invariance of hyperbolic BBM around the starting
point imply that 
$$
\begin{aligned}
\Prob\bigl[ \BB_{\tau}^{\DD} \in \wt D_n \;\text{for some $\tau \in \T$ with }\; |\tau| \ge n\, \bigr] &\le 
\Prob\bigl[ (\BB_{\tau})_{\tau \in \T}  \; \text{ hits }\; 
\mathcal{C}_{-1} \cup \mathcal{C}_{0} \cup \mathcal{C}_1 \bigr] \\
&\le 3 \, \Prob\bigl[ (\BB_{\tau})_{\tau \in \T}  \; \text{ hits }\; \mathcal{C}_{0}\bigr]. 
\end{aligned}
$$
Mapping $\mathcal{C}_0$ back to $\HH$, it becomes the line  
$\{ z \in \HH : \Im z = \sqrt{R_n^2+1} - 1 \}$ (a horocycle with respect to $\im\infty$). 
We get exactly as in the last line of the proof of \cite[Prop. 4]{LS} that
for sufficiently large $n$ 
$$
%\begin{aligned}
\Prob\bigl[ \BB_{\tau}^{\HH} \in D_n \; %&
\text{ for some $\tau \in \T$ with }\; |\tau| \ge n \bigr]%\\
%&
\le 3\, e^{\frac{1+\sqrt{1-8\lambda}}{2}} 
\left(\!\sqrt{C^2 \log n + 1}  - 1 \right)^{-\frac{1+\sqrt{1-8\lambda}}{2}}\,
%\end{aligned}
$$
which tends to $0$ as $n \to \infty$, as proposed.
\end{proof}

\section{The empirical distributions}\label{sec:empdis}

We can interpret hyperbolic BBM as a Markov process on the space of \emph{populations}
in $\DD$, where a population is a finitely supported measure of the form
$$
\Mbf = \sum_{j=1}^n \delta_{z_j}\,,\quad n \in \N\,, \; z_j \in \DD\; \text{not necessarily distinct.}
$$
(Another interpretation is to see this as a multiset; see \cite{KW}.) 

\begin{dfn}\label{def:empdist}
 Starting with one particle at position $z_0\,$, the \emph{occupation measure} 
 of hyperbolic BBM at time $t\ge 0$ is the population
$$
\Mbf_t = \Mbf^{z_0}_t = \sum_{\tau \in \T: |\tau| = t} \delta_{\BB_{\tau}}
$$
and the associated \emph{empirical distribution} is 
$$
\mu_t = \mu_t^{z_0} = \frac{1}{\nn(t)} \Mbf_t\,.
$$
\end{dfn}

Thus, $\mu_t$ is a finitely supported random probability measure on the disk.
Both $\Mbf_t$ and $\mu_t$ depend on the starting point $z_0\,$, and in our construction, 
with $\gm_{z_0}$ as in \eqref{eq:gammaz0}, $\Mbf^{z_0}_t = g_{z_0}\Mbf_t^o$ is the image 
of $\Mbf_t^o$ under the mapping $z \mapsto \gm_{z_0}z$. (That is, for a measure 
$\mu$ on $\overline \DD$ and an isometry $g$ of $\DD$, we have $g\mu(K) = \mu(g^{-1}K)$.)
We shall often omit the starting point in the notation.

Instead of starting BBM at a single point, we may start with an arbitrary initial population 
$\Mbf_0$.
Then the population at time $t$ is
\begin{equation}\label{eq:Markov}
\Mbf_t = \sum_{j=1}^n \Mbf^{z_j}_t\,,\quad \text{if} \quad \Mbf_0 = \sum_{j=1}^n \delta_{z_j}\,,
\end{equation}
where the occupation measures $\Mbf^{z_j}_t$ are independent even when some of the $z_j$ coincide.

We start with an obvious consequence of the continuity of the
distributions of the random variables $\BB_{\tau}\,$.

\begin{rem}\label{rem:distinct} For distinct $\tau, \tau' \in \T$, one has
$\Prob[\BB_{\tau} =  \BB_{\tau'}]=0$. 

(Note that $\tau$ and $\tau'$ are random, but 
the assumption $\tau \ne \tau'$ is inherently deterministic: for the Yule tree, one may 
without change of the probability space consider first the binary tree as a deterministic 
one-complex, where each edge is an interval of length $1$. Then that interval is dilated 
by multiplication with the respective $\Exp_{\lambda}$ edge length random variable. 
Thus each element of $\T$ corresponds uniquely to an element of the deterministic tree, 
and $\tau \ne \tau'$ if and only if
the corresponding elements of the deterministic tree are distinct.)

In particular, for any fixed $t \ge 0$, with probability $1$ the measure 
$\mu_t$ is equidistribution on $\nn(t)$ distinct points.
\end{rem}

Of course this does not imply that it never happens that $\mu_t$ is other than equidistributed.
However, the following is also quite clear.

\begin{lem}\label{lem:rays} Let $\pi$ and $\pi'$ be two distinct rays in the Yule tree. 
Then, with probability $1$, there is a random $t_0 \ge 0$ such that $\BB_{\tau} \ne \BB_{\tau'}$
for all $\tau \in\pi$ and $\tau' \in \pi'$ with $|\tau|, |\tau'| > t_0\,$.  
\end{lem}

\begin{proof}
Let $\vt = \pi \wedge \pi'$. (Compare with the last lines of \S \ref{sec:hypBBM}.) Conditionally
upon the location (value) $\BB_{\vt} = z_0\,$, beyond $\vt$ the two Brownian motions along $\pi$ and $\pi'$
are independent replicas of hyperbolic BM starting at $z_0\,$. Each of them has an a.s. random 
limit $\xi$, resp. $\xi' \in \partial \DD$. Both are distributed according to $\nu_{z_0}\,$; 
see \eqref{eq:poisson} and the preceding lines. Continuity of $\nu_{z_0}$ yields that $\xi \ne \xi'$
almost surely. This implies the statement of the lemma.
\end{proof}

Again, this does not necessarily mean that from some random time onwards, all $\mu_t$ are 
equidistributed. But if we restrict to all $\mu_n\,$, $n \in \N$, then with probability $1$
all of them \emph{are} equidistributed.
For a stronger result than Lemma \ref{lem:rays} concerning branching 
random walks on finitely generated groups,
see {\sc Hutchcroft}~\cite{Hu}. 

\smallskip

In the place of $\mu_t\,$, it will also be useful to use the
discrete random measure 
$$
\Laa_t = \Laa^{z_0}_t = \frac{1}{e^{\lambda t}}\, \Mbf^{z_0}_t\,,
$$
the image of $\Laa^0_t$ under $z \mapsto \gm_{z_0}z$. 
It is not a probability measure, but its expectation  is a deterministic probability 
measure on $\DD$, as the following Lemma shows.

\begin{lem}\label{lem:compact} Let $K \subset \DD$ be compact.
Then the expected number of particles present within $K$ at time $t$ is 
$$
\E_{z_0}\bigl( \Mbf_t(K) \bigr) = e^{\lambda t} \,\Prob_{z_0}[B_t \in K]\,,
$$
where $B_t$ is hyperbolic Brownian motion started at $z_0\,$. Thus,
$$
\E_{z_0}\left( \int_0^{\infty}  \Laa_t(K)\,dt \right) < \infty\,,
$$
and therefore 
$$
\lim_{t \to \infty} \mu_t(K) = 0 \quad \text{almost surely.}
$$
\end{lem}

\begin{proof}
The first identity follows once more from Lemma \ref{lem:many-one}. Finiteness of 
$\E_{z_0}\bigl( \int_0^{\infty}  \Laa_t(K)\,dt \bigr)$ follows from transience
of hyperbolic BM. It implies that the random variable $\int_0^{\infty}  \Laa_t(K)\,dt$
is a.s. finite, whence 
$$
\lim_{t \to \infty} \mu_t(K) = \frac{1}{\ww}
\lim_{t \to \infty} \Laa_t(K) = 0 \quad \text{almost surely,}
$$
where $\ww$ is the martingale limit of Proposition \ref{pro:mart}.
\end{proof}

Thus, the mass of the empirical distributions $\mu_t$ ``disappears at infinity''
in the hyperbolic metric, resp. topology. (More precisely, the $\mu_t$ tend to 0 vaguely within
the hyperbolic disk.) In this sense, the empirical distributions ``do not see'' the difference
between the transient and the recurrent regimes: in the latter, each relatively compact 
open $U \subset \DD$ set is visited
infinitely often, but the proportion of particles visiting $U$ is negligible in the limit. 

We can express the first formula of Lemma \ref{lem:compact} in terms of  the transition 
operator \eqref{eq:Pt} of (non-branching) hyperbolic BM: for a measurable set $K \subset \DD$
and starting point $z_0$, 
$$ 
\E_{z_0}\bigl( \Laa_t(K) \bigr) =  P_t(z_0\,,K)\,.% = \int_{\DD} p_t(z,w)\, d\mm(w),
$$ 
Thus, for a measurable function $f: \DD \to \R$
\begin{equation}\label{eq:fLambda}
\E_{z_0}\left( \int_{\DD} f\, d\Mbf_t \right) = e^{\lambda t}\,P_tf(z_0) \AND
\E_{z_0}\left( \int_{\DD} f\, d\Laa_t \right) = P_tf(z_0)\,,
\end{equation}
as long as the involved integral is well defined.
A \emph{harmonic function}
is a $C^2$-function $h$ on $\DD$ which satisfies $\Lap h \equiv 0$. It is well known that
every harmonic function satisfies
$$
P_t h = h \quad \text{for every }\; t > 0.
$$
Now let $\F_t = \F_t^{\text{BBM}}$ be the sigma-algebra generated by the information 
of hyperbolic BBM up to and
including time $t$. (It projects naturally onto $\F_t^{\text{Yule}}$.) 

\begin{pro}\label{pro:martingale}
If $h$ is a bounded or positive harmonic function, then the family of random variables
$$
\int h \, d\Laa_t = e^{-\lambda t}\sum_{\tau \in \T(t)} h(\BB_{\tau})\,,\quad t \ge 0
$$
is a martingale with respect to the filtration $(\F_t)_{t \ge 0}\,$.
\end{pro}

\begin{proof}
For any bounded or positive measurable function $f : \DD \to \C$, consider its lift 
$\wt f$ to the space of populations:
$$
\wt f(\Mbf) = \int_{\DD} f\, d\Mbf = \sum_{z \in \DD} f(z)\, \Mbf(\{z\})\,,
$$
a finite sum. Let $(\wt P_t)_{t\ge 0}$ be the family of transition operators (i.e. the 
transition semigroup) of hyperbolic BBM, that is, of the Markov process $(\Mbf_t)_{t\ge 0}$.
By \eqref{eq:Markov}, for any population $\Mbf_0 = \sum_{j=1}^n \delta_{z_j}$,
$$
\begin{aligned}
 \wt P_t \wt f(\Mbf_0) &= \E_{\Mbf_0}\bigl( \wt f (\Mbf_t) \bigr)
 = \E_{\Mbf_0}  \left( \sum_z f(z)\, \Mbf_t(\{z\}) \right)\\
 &= \E_{\Mbf_0}  \left( \sum_z f(z) \sum_{j=1}^n \Mbf^{z_j}_t(z) \right)
 = \sum_{j=1}^n \Mbf_0(z_j)\, \E_{z_j}\left( \sum_z f(z) \, \Mbf^{z_j}_t(z) \right)\\
 &= \sum_{j=1}^n \Mbf_0(z_j)\,\E_{z_j}\left( \int f \, d\Mbf^{z_j}_t(z) \right)
 = \sum_{j=1}^n \Mbf_0(z_j) \, e^{\lambda t} \, P_tf(z_j)\\
 &=  e^{\lambda t}\, \wt{P_t f}(\Mbf_0)
\end{aligned}
$$
by \eqref{eq:fLambda}. In particular, if $h$ is a bounded or positive harmonic function on $\DD$
then $\wt P_t \wt h = e^{\lambda t}\, \wt h$ for every $t \ge 0$. This yields that 
$\int h \, d\Laa_t = e^{-\lambda t} \, \int h \, d\Mbf_t$ is a martingale.
\end{proof}

\begin{rem}\label{rem:prob-on-meas} It is good to think of $\Mbf_t$, $\mu_t$ and 
$\sigma_t$ as  random elements of $\mathsf{Meas}(\overline{\DD})$, 
the sigma-compact space  of finite Borel measures on the closed disk, or in 
other words, as measurable mappings 
$\bigl(\Omega\,, \A\,,\Prob\bigr) \to \mathsf{Meas}(\overline{\DD})$ for the probability 
space specified at the beginning of \S \ref{sec:hypBBM}. Thus, their distributions belong to the 
the space of probability measures on $\mathsf{Meas}(\overline{\DD})$.
\end{rem}

\begin{theorem}\label{thm:weakconv}
With probability 1, as $t \to \infty$,  the measures $\Laa_t$ converge weakly
to a Borel measure $\Laa_{\infty}$ on $\partial \DD$, and 
the probability measures $\mu_t$ converge weakly to a probability measure 
$\mu_{\infty}$ on $\partial \DD$. With $\ww$ as in Proposition \ref{pro:mart},
we have 
$$
\Laa_{\infty} = \ww\mu_{\infty}\,, \AND \E_{z_0}(\Laa_{\infty}) = \nu_{z_0}\,,
$$
the measure on $\partial\DD$ whose density with respect to the normalised Lebesgue measure
on the circle is the Poisson kernel \eqref{eq:poisson}. 

In particular, for every $\xi \in \partial \DD$, we have 
$\mu_{\infty}(\{\xi\}) = 0$ almost surely. 
\end{theorem}

\begin{proof}
We use a potential theoretic argument.  Note that the harmonic functions for $\Lap$ are 
the same as the harmonic functions for the Euclidean Laplacian on $\DD$. The
corresponding Dirichlet problem is solvable: given a continuous function $\varphi$ on 
$\partial \DD$, there is a unique harmonic function $h_{\varphi}$ on $\DD$ which provides a
continuous extension of $\varphi$ to the interior of the disk. Indeed, $h_{\varphi}$ is the 
Poisson transform of $\varphi$,
\begin{equation}\label{eq:Poiss}
h_{\varphi}(z) = \int_{\partial\DD} \Poiss(z,\xi)\varphi(\xi)\, d\xi\,.
\end{equation}
Now let $f$ be any continuous function on the closed disk $\overline{\DD}$. Considering
$\Laa_t$ as a measure on $\overline{\DD}$, we have to show that
$\int_{\DD} f\,d \Laa_t$ converges almost surely. Let $\varphi= f|_{\partial \DD}$
be the restriction of $f$ to the unit circle, and let $h_{\varphi}$ be the associated
solution of the Dirichlet problem. This is a bounded harmonic function, so that 
$$
\int_{\DD} h_{\varphi}\,d \Laa_t
$$
converges almost surely by Proposition \ref{pro:martingale}. On the other hand,  
$$
\lim_{|z| \to 1} \bigl(f(z) - h_{\varphi}(z)\bigr) = 0 \quad \text{uniformly in $z$.}
$$
Given $\ep > 0$, take $r \in (0\,,\,1)$ such that $|f(z) - h_{\varphi}(z)| < \ep$
for $|z| \ge r$. Let $K = \{z \in \DD : |z| \le r\}$. By Proposition \ref{pro:mart} and 
Lemma \ref{lem:compact},
$$
\int_{K} (f-h_{\varphi})\,d \mu_t 
= \frac{e^{\lambda t}}{\nn(t)} \int_{K} (f-h_{\varphi})\,d\Laa_t \to 0
\quad \text{almost surely,}
$$
and since $\mu_t$ is a probability measure,
$$
\left|\int_{\DD \setminus K} (f-h_{\varphi})\,d \mu_t \right|
< \ep  \,.
$$
This shows weak convergence. 
The identity 
$\Laa_{\infty} = \ww\mu_{\infty}$ is immediate from 
Proposition \ref{pro:mart}. 
Finally, if $f \in C(\overline \DD)$ then by \eqref{eq:fLambda} and dominated convergence
$$
\begin{aligned}
\E_{z_0}\left( \int_{\DD} f\, d\Laa_{\infty} \right)
&= \lim_{t \to \infty} \E_{z_0}\left( \int_{\DD} f\, d\Laa_t \right)
= \lim_{t \to \infty} P_tf(z_0) = \lim_{t \to \infty} \E_{z_0}\bigl(f(B_t)\bigr) 
= \E_{z_0}\bigl(f(B_{\infty})\bigr)
\\ 
&= \int_{\partial \DD} f\, d\nu_{z_0}\,.
\end{aligned}
$$
In particular, for any $\xi \in \partial \DD$,
$\E_{z_0}\bigl(\Laa_{\infty}(\{\xi\})\bigr) = \nu_{z_0}(\{\xi\})$, so that
$\Laa_{\infty}(\{\xi\}) = 0$ almost surely.
\end{proof}

Note that the latter does not necessarily imply that almost surely, the limit measure carries
no point mass.

The \emph{radial projection} of $\DD$ to the boundary circle $\partial \DD$
is the mapping 
$$
\rad(r e^{\im\psi}) =  e^{\im\psi}\,, \; \text {if }\; 0 < r< 1\,, \AND 
\rad(0) = 1\,.
$$
(The value $\rad(0)$ is of no specific importance and might be chosen arbitrarily.)
We can consider the radial projection of hyperbolic Brownian motion, and the image
of $\mu_t$ under the mapping $\rad$.

\begin{cor}\label{cor:rad}
With probability 1, the probability measures 
$$
\mu_t^{\rad} = \frac{1}{\nn(t)}\sum_{\tau \in \T: |\tau| = t} \delta_{\rad(\BB_{\tau})}
$$
on $\partial \DD$ converge weakly to $\mu_{\infty}\,$, as $t \to \infty\,$. 
\end{cor}

\begin{proof} Let $\varphi \in C(\DD)$, and let $h_{\varphi}$ be its Poisson transform 
\eqref{eq:Poiss}, providing the continuous extension of $\varphi$ to $\overline{\DD}$
which is harmonic on $\DD$. To ``hide'' the discontinuity of $\rad(\cdot)$ at $0$, let
$$
f(z) = \begin{cases}     
\Bigl( h_{\varphi}(z) - \varphi\bigl(\rad(z)\bigr)\Bigr)\, \min \{2|z|, 1\}\,&\text{if }\; z \in \DD,\\
0\,&\text{if }\; z \in \partial\DD
\end{cases}
$$
Then $f \in C(\overline{\DD})$. We get that with probability 1,
$$
\lim_{t \to \infty} \int_{\overline{\DD}} f\, d\mu_t = 0
$$
On the other hand,  
$$
\lim_{t \to \infty} \int_{\overline{\DD}} h_{\varphi}\, d\mu_t 
=  \int_{\partial\DD} \varphi\, d\mu_{\infty}\,.
$$
This concludes the proof.
\end{proof}

In analogy with Corollary \ref{cor:rad},
we have the following for the real parts of hyperbolic BBM in the upper half 
plane model. We omit the very similar proof.

\begin{cor}\label{cor:realpart}  Let
$$
\mu_t^{\Re} 
= \frac{1}{\nn(t)}\sum_{\tau \in \T: |\tau| = t} \delta_{\Re\BB_{\tau}^{\HH}}.
$$
Then, with probability one, $\mu_t^{\Re}$ converges weakly to  
$\mu_{\infty}^{\HH}\,$, as $t \to \infty\,$. 
\end{cor}

Another feature of the ``disappearance'' of the population at infinity 
is the average rate of escape.

\begin{dfn}\label{def:distdist}
 The \emph{(empirical) distance distribution} of hyperbolic BBM at time $t\ge 0$ is
the finitely supported random measure $\mu_t^{\ro}$ on $\R_+$ which is the image of 
$\mu_t$ under the mapping $\DD \ni z \mapsto \ro(z,o)$, that is,
$$
\mu_t^{\ro} = \frac{1}{\nn(t)}\sum_{\tau \in \T: |\tau| = t} \delta_{\ro(\BB_{\tau},o)}.
$$
\end{dfn}

\begin{theorem}\label{thm:normal}
With probability one, $\mu_t^{\ro}$ is asymptotically normal with  mean
$t/2$ and variance $t$. That is, its distribution function satisfies 
$$
\mu_t^{\ro}\bigl(-\infty\,,\tfrac12 t + x \sqrt{t}\,\bigr] 
\to \frac{1}{\sqrt{2\pi}} \int_{-\infty}^x  e^{-s^2/2}\, ds  \quad \text{for every $x \in \R$,}
$$
as $t \to \infty\,$. 
\end{theorem}

Before the proof, we make a detour to the upper half plane model. Recall once more 
from \eqref{eq:im} that the vertical projection  
$$
\BB_{\tau}^{\text{vert}} = -\log \Im \BB_{\tau}^{\HH}\,, \quad \tau \in \T, 
$$
is one-dimensional branching Brownian motion on $\R$ where the 
underlying Euclidean Brownian motion at time $t$ has drift $\frac12 t$
and variance $t$.

For the following, see {\sc Ney}~\cite[Thm. 2]{Ney}, {\sc Kaplan and Asmussen}~\cite[Thm. 3]{KA}
as well as {\sc Biggins}~\cite{Bi2}.

\begin{pro}\label{pro:ney} Let
$$
\mu_t^{\text{vert}} 
= \frac{1}{\nn(t)}\sum_{\tau \in \T: |\tau| = t} \delta_{\BB_{\tau}^{\text{vert}}}.
$$
With probability 1, as $t\to \infty$, the family of discrete probability distributions
$\mu_t^{\text{vert}}$  on $\R$ is asymptotically normal with  mean
$t/2$ and variance $t$. 
\end{pro}

\cite{Ney} only has convergence in mean square, whence in probability. \cite{KA} has almost sure
convergence when the drift of the base Brownian motion is $0$ with obvious extension to non-zero drift;
see also the last section of \cite{Bi2}.

\begin{proof}[Proof of Theorem \ref{thm:weakconv}]
We work with the upper half plane model and indicate this by the superscript $\HH$. 
Let $\Omega_0$ be the event on which $\mu_t^{\HH}$ converges weakly. 
On $\Omega_0\,$, the  measure $\mu_{\infty}^{\HH}$
assigns mass $0$ to the boundary point $\im \infty$, so that it lives on the 
lower boundary line $\R$. 
Let $f \in C^{\infty}(\R)$ be uniformly continuous and set
$$
f_t(x) = f \biggl( \frac{x - \frac12 t}{\sqrt{t}}\biggr).
$$
By the Portmanteau Theorem, we need to show that on $\Omega_0\,$,
$$
\lim_{t \to \infty} \int_{\R} f_t(x) \, d\mu_t^{\ro}(x) 
= \int_{\R} \frac{1}{\sqrt{2\pi}} e^{-x^2/2} f(x)\, dx\,.
$$
Given $\ep > 0$ there is a 
random bound $M < \infty$ such that 
$$
\mu_{\infty}^{\HH}\bigl([-M\,,\,M]\bigr) \ge 1 - \ep\,.
$$
Consider the closed set $U_M = \{z \in \overline\HH : |z| \le M\,,\; \Im z \le 1 \}$,
where $\overline{\HH} = \HH \cup \R \cup \{\im\infty\}$ is the geometric
compactification of the hyperbolic upper half plane.
Then 
\begin{equation}\label{eq:leep}
\limsup_{t \to \infty} \mu_t^{\HH}( \overline \HH \setminus U_M) \le \ep\,.
\end{equation}
We decompose
$$
\int_{\R} f_t(x) \, d\mu_t^{\ro}(x) 
= \int_{\overline{\HH}} f_t\bigl(\ro_{\HH}(z,o)\bigr)\, d\mu_t^{\HH}(z) 
$$
into $\int_{U_M}+\int_{\overline \HH \setminus U_M}$. 
Regarding the second part, we have by \eqref{eq:leep}
$$
\limsup_{t \to \infty} \left|\int_{\overline{\HH} \setminus U_M} 
           f_t\bigl(\ro_{\HH}(z,o)\bigr)\, d\mu_t^{\HH}(z)\right| 
\le \ep\,\|f\|_{\infty} \,.
$$
We apply the same decomposition to 
$$
\int_{\R} f_t(x) \, d\mu_t^{\text{vert}}(x) 
= \int_{\overline{\HH}} f_t\bigl(-\log \Im z\bigr)\, d\mu_t^{\HH}(z) 
$$
and get
$$
\limsup_{t \to \infty} \left|\int_{\overline{\HH} \setminus U_M} 
           f_t\bigl(-\log \Im z\bigr)\, d\mu_t^{\HH}(z)\right| 
\le \ep\,\|f\|_{\infty} \,.
$$
We now consider the difference of the integrals over $U_M\,$, recalling that for 
$\Im z \le 1$ one has $\ro_{\HH}(\im \Im z,o) = - \log \Im z\,$:
$$
\begin{aligned}
&\left| \int_{U_M} 
           f_t\bigl(\ro_{\HH}(z,o)\bigr)\, d\mu_t^{\HH}(z)
- \int_{U_M} f_t\bigl(-\log \Im z\bigr)\, d\mu_t^{\HH}(z) \right|\\ 
&\hspace*{4.5cm}\le 
\int_{U_M} \Bigl|f_t\bigl(\ro_{\HH}(z,o)\bigr) 
   - f_t\bigl(\ro_{\HH}(\im \Im z,o)\bigr) \Bigr|\, d\mu_t^{\HH}(z)\,.
\end{aligned}
$$   
By \eqref{eq:rhodiff}, we have $|\ro_{\HH}(z,o)-\ro_{\HH}(\im \Im z,o)| \le \log(1+M^2)$ on $U_M\,$.
Therefore 
$$
\left|\frac{\ro_{\HH}(z,o)-\frac12 t}{\sqrt{t}} 
- \frac{\ro_{\HH}(\im \Im z,o)-\frac12 t}{\sqrt{t}}\right| \le \frac{\log(1+M^2)}{\sqrt{t}}\,, 
$$
and by uniform continuity of $f$, 
$$
\lim_{t \to \infty} \Bigl|f_t\bigl(\ro_{\HH}(z,o)\bigr) 
   - f_t\bigl(\ro_{\HH}(\im \Im z,o)\bigr) \Bigr| = 0 \quad \text{uniformly on }\; U_M\,.
$$
We infer that on $\Omega_0$
$$
\lim_{t \to \infty} \int_{\R} f_t(x) \, d\mu_t^{\ro}(x) =
\lim_{t \to \infty} \int_{\R} f_t(x) \, d\mu_t^{\text{vert}}(x)\,,
$$
and Proposition \ref{pro:ney} yields the proposed asymptotic behaviour.
\end{proof}

The \emph{average displacement} of hyperbolic BBM is
$$ 
\int_{\DD} \ro(z,o)\,d\mu_t(z) 
= \frac{1}{\nn(t)}\sum_{\tau \in \T: |\tau| = t} \ro(\BB_{\tau},o)
$$ 

\begin{theorem}\label{thm:escape} With probability 1, the average displacement of hyperbolic BBM
has rate of escape
$$
\lim_{t \to \infty} \frac{1}{t}\int_{\DD} \ro(z,o)\,d\mu_t(z) = \frac12\,. 
$$
\end{theorem}

\begin{proof}
In the sample space $\Omega$ of hyperbolic BBM, let $\Omega_1$ be the set on which
$\mu_t^{\ro}$ is asymptotically normal as in Theorem \ref{thm:normal}, and 
$$
\frac{\mathsf{Max}_t}{t} \to r^*
$$
as in Theorem \ref{thm:minmax1}, where
$
\mathsf{Max}_t = \max \bigl\{ \ro(\BB_{\tau},o) : \tau \in \T\,,\; |\tau| = t \bigr\}. 
$ 
We can consider this as providing a random environment for the 
family of probability distributions $\mu_t^{\ro}$, and consider one quenched
case, i.e., a fixed $\omega \in \Omega_1\,$. We can think of the associated
deterministic family $\mu_t^{\ro}$, $t > 0$, as the distributions of a 
family of random variables $X_t$. By asymptotic normality, 
$$
\lim_{t \to \infty} \frac{X_t}{t} = \frac{1}{2}\quad \text{in probability.} 
$$
On the other hand, 
$$
\frac{|X_t|}{t} \le \frac{\mathsf{Max}_t}{t}\,, 
$$
which is bounded. We can apply
the Dominated Convergence Theorem, and denoting by $\E_{\omega}$ the
quenched expectation, we get 
$$
\lim_{t \to \infty} \frac{\E_{\omega}(X_t)}{t} = \frac{1}{2}.
$$
Now 
$$
\E_{\omega}(X_t) = \int_{\R} x \,d\mu_t^{\ro}(x)
= \int_{\DD} \ro(z,o)\,d\mu_t(z)\,,
$$
and this proves the claim.
\end{proof}

\section{The support of the random limit distributions}\label{sec:properties}

Here, we shall collect a few properties on the support of the measures $\mu_{\infty}^{z_0}\,$ plus related 
open questions. Before that, we start with a general fact with a simple proof, communicated
to the author by V.~A.~Kaimanovich during the work on the paper \cite{KW}. The following 
statement is tailored to the present needs.

\begin{pro}\label{pro:nonatomic}
Let $\mathcal{X}$ be a separable compact metric space and $\mathsf{Meas}(\mathcal{X})$ 
be the space of finite Borel measures on $\mathcal{X}$. With the weak*topology, it is separable 
and sigma-compact. Now let $\sigma$ and $\sigma'$ be independent random measures in 
$\mathsf{Meas}(\mathcal{X})$,
with distributions $\boldsymbol{\nu}$ and $\boldsymbol{\nu}'$, respectively. 
The latter are probability measures on $\mathsf{Meas}(\mathcal{X})$, and the distribution of
$(\sigma,\sigma')$ on $\mathcal{X}^2$ is $\boldsymbol{\nu} \otimes \boldsymbol{\nu}'$.

Suppose that the expectation  of $\sigma$,  
that is, the deterministic measure on $\mathcal{X}$ given by
$$
\overline{\sigma}(K) = \int_{\mathsf{Meas}(\mathcal{X})} \sigma(K)\,d\boldsymbol{\nu}(\sigma) 
\quad  (K \subset \mathcal{X}\; \text{compact}) 
$$
on $\mathcal{X}$ is a purely non-atomic measure in $\mathsf{Meas}(\mathcal{X})$. Then, almost surely, 
$\sigma'$ and $\sigma$ share no atoms. 
(``Almost surely'' refers to the probability measure $\boldsymbol{\nu} \otimes \boldsymbol{\nu}'$.)
\end{pro}

\begin{proof}
\emph{Step 1.} Since $\overline{\sigma}$ is non-atomic, for any $x \in \mathcal{X}$
$$
\sigma(x) = 0 \quad \text{for }\; \boldsymbol{\nu} \text{\,-\,almost every }\; 
\sigma \in \mathsf{Meas}(\mathcal{X}). 
$$
\emph{Step 2.} Let $\sigma' \in \mathsf{Meas}(\mathcal{X})$ be fixed (deterministic). Then 
$\boldsymbol{\nu}\text{\,-\,almost every }\; \sigma \in \mathsf{Meas}(\mathcal{X})$ has no 
common atoms with $\sigma'$, that is,
$$
\int_{\mathsf{Meas}(\mathcal{X})} \sigma(x)\sigma'(x)\, d\boldsymbol{\nu}(\sigma) = 0 \quad 
\text{for all }\; x \in \mathcal{X}\,.
$$
\\[3pt]
Indeed, let $x_i\,$, $i \in I$ be the finitely or countably many atoms of $\sigma'$ (if any).
Then the statement of Step 1 holds for each $x_i$.
\\[3pt]
\emph{Conclusion.} For every probability
measure $\boldsymbol{\nu}'$ on $\mathcal{M}(\mathcal{X})$, we can apply Step 2 to get
$$
\int\!\!\!\!\int_{\mathsf{Meas}(\mathcal{X}) \times \mathsf{Meas}(\mathcal{X})} \sigma(x)\sigma'(x)\, 
d\boldsymbol{\nu}(\sigma)\, d\boldsymbol{\nu}'(\sigma') = 0 
\quad \text{for all }\; x \in \mathcal{X}\,,
$$
which is the proposed result.
\end{proof}

For any vertex $\ut \in \{\lt,\rt\}^*$, consider the subtree $\T_{\ut}$ as defined 
in \S \ref{sec:yule}. Its distribution is the same as the one of $\T$.
Recall that $\ut'$ has the role of the ancestor 
$\bal$ in $\T_{\ut}\,$. For the associated martingale according to Proposition 
\ref{pro:mart}, we write
\begin{equation}\label{eq:Wu}
W_{\ut} = \lim_{t \to \infty} \nn_{\ut}(t) e^{-\lambda t}
\end{equation}
for its almost surely existing, finite and positive limit. If $\BB_{\ut'} = z_{\ut'} = z$ then 
$\bigl(g_z^{-1}\BB_{\tau}\bigr)_{\tau \in \T_{\ut}}$ is hyperbolic BBM starting at $o$, 
whose pieces along the edges come from the same construction on our probability space
as for the process starting with the ancestor $\bal$. We have $g_z = \Gb_{\ut'}$ as defined
by \eqref{eq:Gv}. (Simple transitivity
of the group $\Aff$ is useful here, since in this way we can avoid the need to handle
uncountably many different versions.) 
In particular, all the 
empirical distributions of $\bigl(g_z^{-1}\BB_{\tau}\bigr)_{\tau \in \T_{\ut}}$ converge almost
surely, and by continuity of the action of $g_z$,  we have almost surely all the limits
\begin{equation}\label{eq:muu}
%\begin{aligned}
\mu_{\ut, \infty}^{z} %&
= \lim_{t \to \infty} \mu_{\ut, t}^{z}\,, \quad \text{where }\; z=z_{\ut'} = \BB_{\ut'}\;\text{ and }\; %\\
\mu_{\ut, t}^{z} %&
= \frac{1}{\nn_{\ut}(t)} \sum_{\tau \in \T_{\ut}\,:\, |\tau| = t+ |u^-|} 
\delta_{\BB_{\tau}}
%\end{aligned}
\end{equation}
for every $z$ and
for each of the countably many vertices $\ut$ of $\T$. Along with 
$\mu_{\ut, t}^{z}$ and $\mu_{\ut, \infty}^{z}\,$, we also have
$$
\Laa_{\ut, t}^{z} = \nn_{\ut}(t)e^{-\lambda t}\,\mu_{\ut, t}^{z} \AND 
\Laa_{\ut, \infty}^{z} =  W_{\ut}\,\mu_{\ut, \infty}^{z}
$$
with $W_{\ut}$ as in \eqref{eq:Wu}.

\begin{lem}\label{lem:indep}
Let $\ut, \vt \in  \{\lt,\rt\}^*$ be such that none of the two is an ancestor (predecessor) 
of the other.
Consider hyperbolic BBM indexed by the time trees $\T_{\ut}$ and $\T_{\vt}\,$, respectively, and starting at $o$. Then with probability $1$, the limit distributions 
$\mu_{\ut, \infty}^{o}$ and $\mu_{\vt, \infty}^{o}$ of \eqref{eq:muu} share no atoms.
\end{lem}

\begin{proof} The two random measures $\Laa_{\ut, \infty}^{o}$ and $\Laa_{\vt, \infty}^{o}$ 
are independent. 
Now we know from Theorem \ref{thm:weakconv} that the expectation
of the random measure $\Laa_{\ut, \infty}^{o}$ is $\nu_o\,$,
the limit distribution on $\DD$ of hyperbolic Brownian motion starting at $o\,$.
This is normalised Lebesgue measure, whence it has no atoms. Proposition 
\ref{pro:nonatomic} yields that with probability $1$, the measures 
$\Laa_{\ut, \infty}^{o}$ and $\Laa_{\vt, \infty}^{o}$ share no atoms.
Since $W_{\ut}$ and $W_{\vt}$ are almost surely finite and positive, 
and 
$$
\mu_{\ut, \infty}^{o} = \frac{1}{W_{\ut}} \Laa_{\ut, \infty}^{o} 
\AND 
\mu_{\vt, \infty}^{o} = \frac{1}{W_{\vt}} \Laa_{\vt, \infty}^{o},
$$
the lemma follows.
\end{proof}

We can deduce the following.

\begin{theorem}\label{thm:atoms}
With probability $1$, the support of the random limit distribution 
$\mu_{\infty}^{o} = \mu_{\bep, \infty}^{o}$
is infinite.
\end{theorem}

\begin{proof} For any measure $\mu$ on $\partial \DD$, we shall write 
$\mathsf{atoms}(\mu)$ for the set of atoms of  $\mu$, and $\mathsf{supp}(\mu)$
for its support.

Let $\ut \in  \{\lt,\rt\}^*$ and consider hyperbolic BBM according to our construction
of \S \ref{sec:hypBBM} with time tree $\T_{\ut}$ and starting point $\BB_{\ut'} = o$. 
The two children of $\ut$ are $\ut\lt$ and $\ut\rt$. For each of the 
two independent subtrees $\T_{\ut\lt}$ and $\T_{\ut\rt}\,$, the vertex
$\ut$ has the role which the ancestor $\bal$ has in $\T$. Write 
$z_1 = \BB_{\ut}\,$, so that $g_{z_1} = \gb_{\ut}\,$, the random element
of $\Aff$ as described in \S \ref{sec:hypBBM}.
Denote by $\Mbf_{\ut\lt,t}^{z_1}$ and $\Mbf_{\ut\rt,t}^{z_1}$ the occupation measures
at time $t$ (that is, distance $t$ from $u$ in $\T_{\lt}$ and $\T_{\rt}$),
respectively. 
Then
$$
\Mbf_{\ut,t}^o = \Mbf_{\ut\lt,t-\ell_{\ut}}^{z_1} + \Mbf_{\ut\rt,t-\ell_{\ut}}^{z_1}\,.
$$
Dividing by $\nn_{\bep}(t)$ and letting $t \to \infty\,$,
$$
\mu_{\ut,\infty}^o 
= \cb_{\ut\lt} \,\mu_{\ut\lt,\infty}^{z_1}
+ \cb_{\ut\rt}\,\mu_{\ut\rt,\infty}^{z_1},
\quad \text{where}\quad \cb_{\ut\lt} = \frac{W_{\ut\lt}}{W_{\ut}} \, e^{-\lambda\ell_{\ut}}
\;\text{ and }\; \cb_{\ut\rt} = \frac{W_{\ut\rt}}{W_{\ut}} \, e^{-\lambda\ell_{\ut}}\,,
$$
a non-trivial convex combination of two probability measures. 
We can also rewrite this as
$$
\gb_{\ut}^{-1}\mu_{\ut,\infty}^o 
= \cb_{\ut\lt} \,\mu_{\ut\lt,\infty}^{o}
+ \cb_{\ut\rt}\,\mu_{\ut\rt,\infty}^{o}\,.
$$
From this identity and Lemma \ref{lem:indep} we infer that
$$
\mathsf{atoms}(\mu_{\ut,\infty}^o) 
= \gb_{\ut}\mathsf{atoms}(\mu_{\ut\lt,\infty}^o)  \stackrel{{\,}_{_{_{_+}}}}{\cup} 
%\mathbin{\dot{\cup}}
\gb_{\ut}\mathsf{atoms}(\mu_{\ut\rt,\infty}^o)
$$
is almost surely a disjoint union. 

Recursively, we get the following convex combination for each $n$:
\begin{equation}\label{eq:meassum}
\mu_{\bep,\infty}^o = 
\sum_{\vt \in \{\lt,\rt\}^n} \cb^{(n)}_{\vt}\, \Gb_{\vt'}\, \mu_{\vt,\infty}^o
\end{equation}
with positive random constants $\cb^{(n)}_{\vt}$ and $\Gb_{v'}$ as defined by \eqref{eq:Gv}, and 
$$
\mathsf{atoms}(\mu_{\ut,\infty}^o) = %\mathbin{\dot{\bigcup}}
\ {\bigcup}_{\vt \in \{\lt,\rt\}^n}^{\!\!\!\!\!\!\!+}
%\stackrel{{\,\,}_{_{_{_{_{_+}}}}}}{\bigcup}_{\vt \in \{\lt,\rt\}^n}  
\Gb_{\vt'}\, \mathsf{atoms}(\mu_{\vt,\infty}^o) \quad \text{almost surely.}
$$
If $\mathsf{supp}(\mu_{\ut,\infty}^o)$ is finite then the support of the measure coincides with
the set of its atoms, and there must be $\vt$ such that $\mu_{\vt,\infty}^o$ has no atoms.
But along with $\mu_{\bep,\infty}^o\,$, by \eqref{eq:meassum} all $\mu_{\vt,\infty}^o$
have finite support consisting only of atoms, a contradiction.
\end{proof}

\begin{dfn}\label{def:limset}
The \emph{limit set} of hyperbolic BBM is the random set $\Ls$ of accumulation points 
of $(\BB_{\tau})_{\tau \in \T}$ on the unit circle $\partial \DD$.
\end{dfn}

{\sc Lalley and Sellke}~\cite{LS} have proved the following significant result.

\begin{theorem} \cite{LS} \label{thm:LS}
In the unit circle with the arclength measure, there is the following dichotomy.
\\[3pt]
\emph{(i)} For $\lambda \le 1/8$, with probability $1$, $\Ls$ is a 
Cantor set (totally disconnected and
perfect), and its Hausdorff dimension is $(1-\sqrt{1-8\lambda}\,)/2$.
Furthermore, $\Ls$ is contained in a proper sub-arc of $\DD$.
\\[3pt]
\emph{(ii)} For $\lambda > 1/8$, with probability $1$, $\Ls = \partial \DD$.
\end{theorem}

It is clear that
$\;
\mathsf{supp}(\mu_{\infty}^o) \subseteq \Ls\,.
\;$
At this point several interesting questions arise.

\begin{ques}\label{ques:tions} (a) Is $\mathsf{supp}(\mu_{\infty}^o) \subset \Ls$ 
properly$\,$?
\\[3pt]
(b) What is the Hausdorff dimension of $\mu_{\infty}^o$ (i.e. the smallest Hausdorff 
dimension of a set with full $\mu_{\infty}^o$-measure)$\,$? Is it strictly smaller than the 
Hausdorff dimension of $\Ls\,$?
\\[3pt]
(c) Is $\mu_{\infty}^o$ purely non-atomic$\,$? Is it absolutely continuous with respect to
Lebesgue (arc) measure in the supercritical regime $\lambda > 1/8\,$?
\end{ques}

\noindent
\textbf{Note added in proof.} After an early arXiv version of the present paper,
D. Geldbach \cite{Ge} has given striking answers to these questions.
\\[3pt]
(i) The Hausdorff dimension of $\mathsf{supp}(\mu_{\infty}^o)$ is $\min \{2\lambda, 1 \}$. 
This is not necessarily the same as the Hausdorff dimension of the measure itself.
In any case, the result may support a 
suggestion of V. A. Kaimanovich (oral communication) that the dimension of 
$\mu_{\infty}^o$ could be the quotient of the entropy rate and the escape rate of 
the empirical distributions $\mu_t\,$.  

In particular, for $\lambda < 1/2$, one has $\mathsf{supp}(\mu_{\infty}^o) \subset \Ls$ strictly.
\\[3pt]
(ii) The measure $\mu_{\infty}^o$ is always purely non-atomic.
\\[3pt]
(iii) For $\lambda > 1/2$, it is absolutely continuous with respect to
Lebesgue measure.

\end{document}